\documentclass[12pt,reqno]{amsart}
\usepackage[margin=1.25in]{geometry} 
\usepackage{graphicx}
\usepackage{amssymb}
\usepackage{dsfont}
\usepackage{aliascnt}
\usepackage{doi}
\usepackage{epstopdf}
\usepackage{tikz}
\usetikzlibrary{3d}
\usetikzlibrary {arrows.meta}

\newcommand\myarrowL[3]{
\draw[arrows = {-Stealth[length=6.5pt, inset=4pt, width=7pt]}, line width=0.55pt, color=#3] #1 to[bend left=25] #2;}

\usepackage{esint}

\newtheorem{theorem}{Theorem}[section]

\newaliascnt{corx}{thmx}

\aliascntresetthe{corx}

\newaliascnt{lemma}{theorem}
\newtheorem{lemma}[lemma]{Lemma}
\aliascntresetthe{lemma}

\newaliascnt{proposition}{theorem}
\newtheorem{proposition}[proposition]{Proposition}
\aliascntresetthe{proposition}

\newaliascnt{corollary}{theorem}
\newtheorem{corollary}[corollary]{Corollary}
\aliascntresetthe{corollary}

\newaliascnt{conjecture}{theorem}
\newtheorem{conjecture}[conjecture]{Conjecture}
\aliascntresetthe{conjecture}

\newaliascnt{question}{theorem}

\aliascntresetthe{question}

\theoremstyle{definition}
\newtheorem*{definition*}{Definition}
\newtheorem*{example*}{Example}
\newtheorem*{remark*}{Remark}

\newcommand{\arxiv}[1]{%
\href{https://arxiv.org/abs/#1}{ArXiv:#1}}

\newcommand{\B}{{\mathbb B}}
\newcommand{\e}{\varepsilon}

\newcommand{\R}{{\mathbb R}}

\newcommand{\Rn}{{{\mathbb R}^n}}
\newcommand{\Rnp}{{{\mathbb R}^{n+1}}}

\newcommand{\Sphn}{{{\mathbb S}^n}}
\newcommand{\xh}{\hat{x}}
\newcommand{\yh}{\hat{y}}
\newcommand{\xih}{\hat{\xi}}
\newcommand{\Z}{{\mathbb Z}}

\newcommand{\s}{{\text{\small $\star$}}}
\newcommand{\deltastar}{\Delta^{\! \s}}
\newcommand{\deltastart}{\Delta^{\! \s T}}

\DeclareMathOperator{\essinf}{{ess{\,}inf}}

\DeclareMathOperator{\dist}{dist}
\DeclareMathOperator{\diam}{diam}

\DeclareMathOperator{\capzero}{Cap_0}
\DeclareMathOperator{\capone}{Cap_1}
\DeclareMathOperator{\captwo}{Cap_2}

\DeclareMathOperator{\capq}{Cap_\mathit{q}}

\DeclareMathOperator{\capntwo}{Cap_{\mathit{n}-2}}
\DeclareMathOperator{\capnone}{Cap_{\mathit{n}-1}}

\title[Logarithmic vs Newtonian capacity]{Toward P\'{o}lya and Szeg\H{o}'s conjecture for logarithmic vs Newtonian capacity}
\author{Carrie Clark and Richard S. Laugesen}
\address{University of Illinois, Urbana IL 61801, USA}
\email{carrieclark435@gmail.com}
\address{University of Illinois, Urbana IL 61801, USA}
\email{Laugesen@illinois.edu}
\subjclass[2020]{Primary 31B15; Secondary 31A15}
\keywords{Riesz capacity, logarithmic capacity, Newtonian capacity, Neumann Green function, symmetrization}
\hypersetup{
  pdftitle={Toward Polya and Szego's conjecture for logarithmic vs Newtonian capacity},
  pdfauthor={Carrie Clark and Richard S. Laugesen},
  pdfsubject={Riesz capacities, insulated strip energies, and symmetrization},
  pdfkeywords={Riesz capacity, logarithmic capacity, Newtonian capacity, Neumann Green function, symmetrization}
}

\begin{document}

\begin{abstract}
P\'{o}lya and Szeg\H{o} conjectured in 1945 that the disk achieves the largest electrostatic capacity among all planar sets with given logarithmic capacity. This question generalizes naturally to balls and Newtonian capacities in higher dimensions.

The conjecture remains open. This paper establishes a conditional result: the conjecture would follow if among all lower dimensional sets in a horizontal strip with Neumann boundary conditions, the transverse component of the Dirichlet integral of the equilibrium potential could be shown minimal for the lower dimensional ball. As supporting evidence toward this latter property, the paper shows extremality of the ball for $L^p$-norms of the potentials, via Baernstein star-function symmetrization techniques.
\end{abstract}

\maketitle

\section{\bf Introduction}


P\'{o}lya and Szeg\H{o} conjectured in 1945 that under the passage from logarithmic to Newtonian capacity, the disk retains more capacity than any other planar set:
\begin{equation} \label{PS2}
\capone(K) \leq \frac{2}{\pi} \capzero(K)
\end{equation}
for all compact planar sets $K$, with equality when $K$ is a disk \cite[Conjecture (1.3)]{PS45}. Here $\capone(K)$ is the Newtonian or electrostatic capacity of $K$ regarded as a set in $\R^3$ and $\capzero(K)$ is its logarithmic capacity in $\R^2$.

The initial challenge posed by the conjecture is to somehow relate the two capacities that live in different dimensions. A recent paper of ours met that challenge by developing a $1$-parameter family of insulated strip energies that yield the logarithmic and Newtonian energies as limiting cases \cite{CL26}.

The next challenge, forming the focus of the current paper, is to obtain a tractable formula for the derivative of the strip energy with respect to the thickness, a task made more difficult by the logarithmic divergence of the strip potential at infinity. With the derivative formula in hand, we then show that P\'{o}lya and Szeg\H{o}'s conjecture and a higher dimensional analogue would follow from a conjectured sharp ``transverse component of the Dirichlet energy'' inequality between competing harmonic potentials in an infinite strip.

The third and remaining challenge for the conjecture will be to prove the transverse energy inequality. While that remains a task for the future, there is reason to believe it is within reach. Notably, we establish in this paper a sharp symmetrization inequality between the $L^p$-norms of the competing potentials.

\subsection*{History of the problem} P\'{o}lya and Szeg\H{o} restated the conjecture in their celebrated book “Isoperimetric Inequalities in Mathematical Physics” \cite{PS51} in 1951; see page 10 and item {\#}17 on page 18. Earlier, in 1931, they had proved a weaker inequality $\capone(K) \leq \capzero(K)$ \cite[p.{\,}19]{PS31}. As far as we know the best partial result is due to Szeg\H{o} \cite{S52,S56}, who showed using conformal mapping techniques that if the compact set $K \subset \R^2$ is connected then
\[
\capone(K) \leq (1.07) \frac{2}{\pi} \capzero(K) .
\]
Thus Szeg\H{o} got within 7\% of the conjectured optimizer. His conformal mapping approach does not generalize to higher dimensions.

The conjectured inequality \eqref{PS2} extends to a whole family of extremal conjectures for pairs of Riesz capacities. We recently explored this family and proved some of the claims \cite{CL25b}. Because endpoint cases of Riesz capacity include the diameter and area of the set \cite{CL25a}, this family of problems reveals that P\'{o}lya and Szeg\H{o}'s conjecture is closely related to the classical isodiametric inequality and to Watanabe's theorem \cite{W83} that says the disk minimizes Riesz capacity among sets of given area.

Equality does indeed hold in \eqref{PS2} for the unit disk, because the disk has logarithmic capacity $1$ and Newtonian capacity $2/\pi$. (Newtonian capacity is normalized to equal $1$ for the unit ball.) At an opposite extreme, an interval has positive logarithmic capacity but Newtonian capacity $0$, and so the left side of inequality \eqref{PS2} has no positive lower bound in general.

\subsection*{Higher dimensional P\'{o}lya--Szeg\H{o} conjectures} Write $\capntwo(K)$ for the Newtonian capacity of a compact set $K \subset \Rn$, so that $\capnone(K)$ is the Newtonian capacity when $K$ is regarded as lying in $\Rnp$. Logarithmic capacity is $\capzero$. Definitions are given in \autoref{sec:capacities}. Each of these capacities scales linearly, with $\text{Cap}(sK) = s \, \text{Cap}(K)$ for all $s>0$.

Physically, the Newtonian energy $1/\capone(K)$ in $3$ dimensions represents the electrostatic energy (work) required to bring a unit of positive charges in from infinity and place the charges on $K$, working against their mutual repulsion.

P\'{o}lya and Szeg\H{o}'s conjecture extends naturally to all dimensions. We formulate it in terms of maximizing the scale-invariant ratio of capacities.
\begin{conjecture}[Ratio of Newtonian capacities] \label{conj:riesz}
Let $n \geq 2$. Among all compact $K \subset \Rn$ with positive $(n-2)$-capacity, the closed $n$-ball maximizes the ratio
\[
\frac{\capnone(K)}{\capntwo(K)} .
\]
\end{conjecture}
\subsubsection*{Remarks}
1. If the conjecture holds then the maximum value of the ratio would be $2/\pi$ when $n=2$, by evaluating the Newtonian and logarithmic capacities of the unit disk, and when $n \geq 3$ it would be
\[
\left. \left( \frac{\Gamma(n/2)}{\Gamma(1/2) \Gamma((n+1)/2)} \right)^{\! \! 1/(n-1)}
\right/
2 \left( \frac{\Gamma(n/2) \Gamma((n-1)/2)}{\Gamma(1/2) \Gamma(n-1)} \right)^{\! \! 1/(n-2)}
\]
by using known capacity formulas for the ball $\overline{\B}^n$ (see \cite[Appendix A]{CL25a}). Choosing $n=3$ in this formula gives $1/\sqrt{2}$, and so in particular the conjecture asserts for $K \subset \R^3$ that
\[
\captwo(K) \leq \frac{1}{\sqrt{2}} \capone(K) .
\]
2. The $1$-dimensional case of the conjecture is known and straightforward to prove: take $(n,p,q)=(1,-1,0)$ in \cite[Proposition 3]{CL25b}.

\subsection*{The transverse energy conjecture}
The key to the capacity conjecture, in our approach, is the conjecture below on harmonic functions in the horizontal strip
\[
S=\Rn \times (-1,1) = \{ (x,z) : x \in \Rn, -1<z<1 \} , \qquad n \geq 2 .
\]
The strip, depicted in \autoref{fig:strip}, arises as an intermediate domain between $\Rn$ and $\Rnp$, where we identify $\Rn$ with the central hyperplane $\Rn \times \{ 0 \}$. The harmonic functions $U$ and $V$ in what follows are equilibrium potentials for the set $K$ and ball $B$ in the strip, subject to Neumann conditions on the boundary.

The notion of inner capacity zero is defined in \autoref{sec:background}.

\begin{figure}
\begin{center}
\begin{tikzpicture}[x={(2cm,0cm)},y={(0.8cm,0.4cm)},z={(0cm,1cm)},scale=1.8]

\fill[gray!9] (1.45,1,1) -- (1.45,1,-1) -- (1.45,-1,-1) -- (-1.2,-1,-1) -- (-1.2,-1,1) -- (-1.2,1,1) -- cycle;
\draw[-,dashed] (1.45,1,1) -- (1.45,1,-1);
\draw[-,dashed] (1.45,-1,1) -- (1.45,-1,-1);
\draw[-,dashed] (-1.2,1,0.18) -- (-1.2,1,-1);
\draw[-,dotted,thick] (-1.2,1,0.2) -- (-1.2,1,1);
\draw[-,dashed] (-1.2,-1,1) -- (-1.2,-1,-1);

\draw[->] (-0.73,0,0) -- (0.95,0,0);
\draw[->] (0,1.2,0) -- (0,-1.2,0) ;
\draw[-] (0,0,0) -- (0,0,-1);
\draw[dotted, thick] (0,0,-1.4) -- (0,0,-1);
\draw[-] (0,0,-1.4) -- (0,0,-2);
\draw[-] (-0.012,0,1) -- (0.012,0,1) node[right] {$z=1$};
\draw[-] (-0.012,0,-1) -- (0.012,0,-1) node[right] {$z=-1$};
 \filldraw[fill opacity=0.9,gray!37,draw=black] plot[smooth,samples=50,domain=0:356,variable=\t] ({0.55*(cos(\t)+0.1*sin(5*\t)+0.05},{0.72*(sin(\t)+0.02*sin(2*\t))}) -- cycle;

\draw[-] (0,0,0) -- (0,0,0.6);
\draw[dotted, thick] (0,0,0.6) -- (0,0,1);
\draw[->] (0,0,1) -- (0,0,2);

\draw (1.1,0,1.7) node {\large$S=\Rn\times (-1,1)$};
\draw (0.27,0,0.13) node {$K\subset\Rn$};
\draw (-1.2,0,0.3) node {$U=E$};
\draw (-1.2,0,1.4) node[above] {{\large$\frac{\partial U}{\partial \nu}$}$\,=0$};
\draw (-1.2,0,-2) node[above] {{\large$\frac{\partial U}{\partial \nu}$}$\,=0$};

\myarrowL{(-1,0,1.65)}{(-0.3,0,1)}{black}
\myarrowL{(-1.03,0,-1.6)}{(-0.3,0,-1)}{black}

\myarrowL{(-1.,0,0.3)}{(-0.15,0,0.05)}{black}

\fill[fill opacity=0.1,gray] (1.45,1,1) -- (1.45,-1,1) -- (-1.2,-1,1) -- (-1.2,1,1) -- cycle;
\fill[fill opacity=0.1,gray] (1.45,1,-1) -- (1.45,-1,-1) -- (-1.2,-1,-1) -- (-1.2,1,-1) -- cycle;
\end{tikzpicture}
\end{center}
\caption{\label{fig:strip}
The compact set $K$ for \autoref{conj:neumann} lies within the central hyperplane of the strip $S=\Rn\times(-1,1)$. The function $U$ is harmonic in $S\setminus K$ and satisfies the Dirichlet boundary condition $U=E$ on $K$ and Neumann condition $\partial U /\partial \nu = 0$ on the upper and lower strip boundaries where $z=\pm 1$.
}
\end{figure}
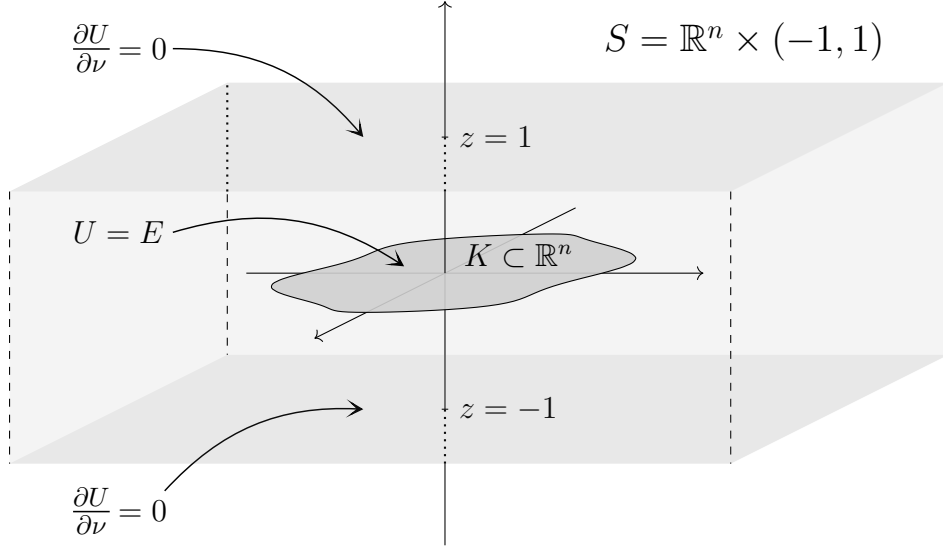
\begin{conjecture}[Ball minimizes transverse component of energy] \label{conj:neumann}
Fix $n\geq 2$ and $E \in \R$, and if $n \geq 3$ then further suppose $E>0$. Let $K \subset \Rn$ be compact. Suppose $U$ is a superharmonic function on $S$ that is harmonic on $S \setminus K$ with $U=E$ on $K$ (except perhaps on a subset of inner $(n-1)$-capacity zero) and $\partial U/\partial \nu = 0$ on $\partial S$, and that as $r=|x| \to \infty$ one has
\[
\begin{split}
U(x,z)
& =
\begin{cases}
\dfrac{1}{r^{n-2}} + O \! \left( \dfrac{1}{r^{n-1}} \right) & \text{if $n \geq 3$,} \\[1em]
  \log \dfrac{1}{r} + O \! \left( \dfrac{1}{r} \right) & \text{if $n=2$,}
\end{cases}
\\
\nabla U(x,z) & = - \frac{b_n}{r^{n-1}} \, (x/r,0) +O\!\left( \frac{1}{r^{n}}\right),
\end{split}
\]
where $b_2=1$ and $b_n=n-2$ for $n \geq 3$.

Let $B \subset \Rn$ be a closed ball. If $V$ is superharmonic in $S$ and harmonic on $S \setminus B$ with $V=E$ on $B$ (except perhaps on a subset of inner $(n-1)$-capacity zero) and $\partial V/\partial \nu = 0$ on $\partial S$, and if $V$ and $\nabla V$ satisfy the above decay estimates when $|x| \to \infty$, then $V$ has less energy than $U$ in the transverse (vertical) direction, meaning
\begin{equation} \label{eq:conjconclusion}
\int_S \left( \frac{\partial U}{\partial z} \right)^{\!\! 2} \,dxdz \geq \int_S \left( \frac{\partial V}{\partial z} \right)^{\!\! 2} \, dxdz .
\end{equation}
Equality holds if and only if the compact set $K \subset \Rn$ equals the union of a translate of $B$ and a set of inner $(n-1)$-capacity zero.
\end{conjecture}

\subsubsection*{Remark}
The transverse energy conclusion \eqref{eq:conjconclusion} in the conjecture turns out to be equivalent in dimensions $n \geq 3$ to the reverse inequality on the horizontal component of the energy, namely
\[
\int_S |\nabla_{\! x} U|^2\,dxdz \leq \int_S |\nabla_{\! x} V|^2 \, dxdz .
\]
The justification is that the total Dirichlet integrals agree, meaning
\[
\int_S |\nabla U|^2\,dxdz = \int_S |\nabla V|^2 \, dxdz ,
\]
where this agreement can be seen by a formal argument with Green's formula or else rigorously by using \autoref{StripDirichletReciprocal}. In $2$ dimensions, one can proceed similarly except taking the difference of the two Dirichlet integrals in order to obtain a convergent expression (using \autoref{StripDirichlet2}).

\subsection*{Main result}
P\'{o}lya and Szeg\H{o}'s capacity conjecture would follow from the transverse energy conjecture.
\begin{theorem} \label{th:main}
\[
\text{\autoref{conj:neumann} (transverse energy)} \quad \Longrightarrow \quad \text{\autoref{conj:riesz} (capacity ratio)}
\]
\end{theorem}

\subsubsection*{Strategy of the proof} We connect the two different capacities in the P\'{o}lya--Szeg\H{o} conjecture through a physically natural one-parameter family that we investigated recently in \cite{CL26}. Namely, we enclose the conductor $K$ in an insulated strip in $\Rnp$ and vary the strip thickness from $0$ to $\infty$. By taking the thickness to $\infty$ one recovers the $(n-1)$-capacity, which is Newtonian capacity in $\Rnp$. At the other extreme, letting the thickness tend to $0$ yields the $(n-2)$-capacity, which is Newtonian capacity in $\Rn$.

To follow this pathway and prove the theorem requires additional stepping stones.
\begin{itemize}
	\item A sharp inequality on moments of the Newtonian equilibrium measure, provided by our recent work \cite{CL24}, allows us to compare the strip energies for $K$ and the ball $B$ as the strip thickness tends to infinity.
	\item A variational formula enables the strip energy to be differentiated with respect to the thickness. This formula, which we develop in \autoref{EKderivDirichlet}, is expressed in terms of the $L^2$ norm of the transverse derivative of the equilibrium potential. The case $n=2$ must be treated with particular care, due to the logarithmic divergence of the potential at infinity and slow decay of the gradient.
	\item Relying on the energy derivative formula, \autoref{conj:neumann} gives a sharp inequality between the derivatives of the strip energy for $K$ and for the ball $B$, at each thickness where their energies agree. Thus we may track those energies from $t=\infty$ all the way down to $t=0$, yielding the desired capacity comparison for \autoref{conj:riesz}.
\end{itemize}
The limiting process as $t \to 0$ turns strict inequalities into nonstrict, and so our method is unable to analyze the equality case in \autoref{conj:riesz}.

\subsection*{Outline of the paper} Newtonian and logarithmic energies and capacities are defined in \autoref{sec:capacities}. In \autoref{sec:stripenergy} we recall the strip kernels, which are Green functions satisfying Neumann conditions at the boundary of the strip, and then define their corresponding energies. The crucial derivative formula for the energy with respect to the strip thickness is stated in \autoref{sec:stripenergy} and proved in \autoref{sec:dirichlet3} and \autoref{sec:dirichlet2}. The main result \autoref{th:main} is then derived in \autoref{sec:stripmin}.

\autoref{sec:starfninequalities} and \autoref{sec:proofofJvustar} develop a comparison result that supports \autoref{conj:neumann} in the sense that the harmonic function associated with the ball has larger $L^p$-norm on each horizontal slice than the harmonic function associated with the arbitrary set.

\autoref{sec:background} establishes foundational facts about equilibrium potentials in the Neumann strip.

\subsection*{Related problems and literature} Extremal results in the literature generally minimize capacity subject to a geometric constraint, rather than a constraint on another capacity as in the current paper. The Poincar\'{e}--Carleman--Szeg\H{o} theorem \cite{PS51} asserts that the set of given volume that minimizes Newtonian capacity $\capntwo$ is the ball, and that in $2$ dimensions the set of given area minimizing logarithmic capacity $\capzero$ is the disk. A beautiful polygonal analogue was proved by Solynin and Zalgaller \cite{SZ04}, that among $N$-sided polygons with given area, logarithmic capacity is minimal for the regular $N$-gon. Extremal problems for conformal modulus under extension of a doubly connected domain were studied  by Laugesen \cite{L93} and Solynin \cite{S92}.

Generalizing from Newtonian to Riesz capacity $\capq$, the ball was shown to minimize among sets of given volume when $n-2<q<n$ by Watanabe \cite[p.{\,}489]{W83}; see also Betsakos \cite{B04a,B04b} and M\'{e}ndez--Hern\'{a}ndez \cite{MH06}. Lastly, for the plane again, Schwarz lemma type extremal results for logarithmic capacity under conformal mapping have been obtained by Burckel \emph{et al.}\ \cite{BMMPR08} and Betsakos and Pouliasis \cite{BP13}.

Another well known capacity conjecture by P\'{o}lya and Szeg\H{o} that still remains open is the claim that among convex sets in $3$-space with given surface area, the Newtonian capacity $\capone$ is minimal for the degenerate double-sided disk. Jerison's solution of the Minkowski problem for Newtonian capacity \cite{J96} might be relevant to this problem. Excellent partial results are due to Bucur \emph{et al.}\ \cite{BFL12}, I. Fragal\`{a} \emph{et al.}\ \cite{FGP11}, and Xiao \cite{LX25,X17}.

\section{\bf Newtonian and logarithmic energy and capacity: definitions and properties}
\label{sec:capacities}

For the following facts, see \cite[Lemma 4.1.3, Theorem 4.4.5, Theorem 4.4.8]{BHS19}.

Consider a compact nonempty set $K$ in a Euclidean space. The dimension of the space does not appear in the following definitions and so need not be specified. When $q > 0$, the Riesz $q$-energy of $K$ is
\[
V_q(K) = \min_\mu \int_K \! \int_K |x-y|^{-q} \, d\mu(x) d\mu(y) ,
\]
where the minimum is taken over all probability measures on $K$, that is, positive unit Borel measures. The minimum is attained by an ``equilibrium" measure. The energy is positive or $+\infty$, and if the energy is finite then the equilibrium measure is unique, in which case we call it the $q$-equilibrium measure of $K$. For the empty set we define $V_q(\emptyset)$ to equal $+\infty$.

When $q=0$ one considers instead the logarithmic energy
\[
V_{log}(K) = \min_\mu \int_K \! \int_K \log \frac{1}{|x-y|} \, d\mu(x) d\mu(y) ,
\]
with the minimum taken over probability measures on $K$. The minimum is attained by an equilibrium measure. The energy is greater than $-\infty$ since $|x-y|$ is bounded on $K \times K$. If the energy is less than $+\infty$ then the equilibrium measure is unique and is called the $0$-equilibrium measure. For the empty set, define $V_{log}(\emptyset)=+\infty$.

The Riesz capacity of $K$ is
\begin{equation*}
\capq(K) =
\begin{cases}
V_q(K)^{-1/q} , & q > 0 , \\
\exp(-V_{log}(K)) , & q = 0 .
\end{cases}
\end{equation*}
The $0$-capacity is known also as logarithmic capacity.

The Riesz capacity is positive if and only if the energy is finite. The definition ensures that capacity is monotonic with respect to set inclusion ($K_1 \subset K_2$ implies $\capq(K_1) \leq \capq(K_2)$) and that it scales linearly, with $\capq(sK)=s \capq(K)$ whenever $s>0$. Capacity is also monotonically decreasing with respect to the Riesz parameter $q$ (see \cite[Theorem 1.1]{CL25a}), although we will not need that fact.

In this paper we work with exponents $q=n-2$ and $q=n-1$, corresponding to Newtonian kernels for $\Rn$ and $\Rnp$ respectively. The logarithmic case arises when $n=2$ and $q=n-2=0$.

\section{\bf Strip energy and its derivative with respect to the thickness}
\label{sec:stripenergy}

Let us recall from our recent paper \cite{CL26} the definition of the Newtonian-type kernel with Neumann boundary conditions on the infinite strip
\[
S(t)=\Rn \times (-t,t) .
\]
Fix $n \geq 2$ and let $t>0$. For $j \in \Z$, the transformation
\begin{equation} \label{eq:rhodefn}
\rho_j(y,w)=(y,2tj+(-1)^j w) , \qquad y \in \Rn, \ w \in \R ,
\end{equation}
translates points vertically when $j$ is even and acts by reflection and translation when $j$ is odd, as illustrated in \autoref{fig:yhat}. Although the mapping $\rho_j$ depends on $t$, we suppress that dependence for notational simplicity. Notice $\rho_0(y,w)=(y,w)$ is the identity map.
\begin{definition*}[Strip kernel with Neumann boundary conditions]
First suppose $n \geq 3$. Define a kernel on the closure of the strip by the method of reflections, letting
\begin{equation} \label{eq:rieszkernelq}
G_t(\xh,\yh) = \sum_{j \in \Z} \frac{1}{| \xh-\rho_j(\yh) |^{n-1}}
\end{equation}
when $\xh=(x,z), \yh=(y,w) \in \overline{S(t)}$, that is, when $x,y \in \Rn$ and $z,w\in[-t,t]$. One checks easily that if $\xh \neq \yh$ then $\xh \neq \rho_j(\yh)$ for all $j$. Clearly the series for $G_t$ is greater than zero everywhere and converges locally uniformly wherever $\xh \neq \yh$, since $n-1 >1$.

In order to define the kernel when $n=2$, renormalize the formula by subtracting a summand that cancels the leading order divergent term: we define
\begin{equation}\label{eq:rieszkernelone}
G_t(\xh,\yh)
= \frac{1}{|\xh-\yh|}+ \sum_{j\neq 0}\left( \frac{1}{|\xh-\rho_j(\yh)|} -\frac{1}{|2tj|} \right) + \frac{1}{t}(\gamma-\log (4t))
\end{equation}
for $\xh, \yh \in \overline{S(t)}$, where $\gamma \simeq 0.577$
is the Euler--Mascheroni constant. The series converges locally uniformly where $\xh \neq \yh$, since $1/|\xh-\rho_j(\yh)|$ behaves like $1/|2tj| + O(1/j^2)$ for large $|j|$. Thus in particular, although we do not know that the kernel is positive when $n=2$, it is certainly bounded below provided $\xh$ and $\yh$ lie in some bounded subset of the strip.
\qed
\end{definition*}

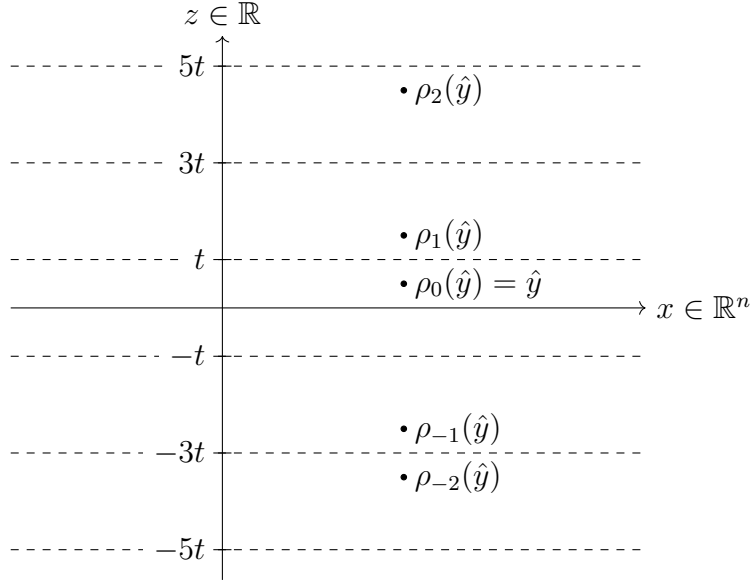
\begin{figure}
\begin{center}
\begin{tikzpicture}[scale=0.8]
\def\t{0.8};
\def\xmin{-3.5};
\def\xmax{7};
\def\zmin{-5*\t-0.5};
\def\zmax{5*\t+0.5};
\def\y{3};
\def\w{\t/2};
\draw[dashed] (\xmin,\t) -- (\xmax,\t);
\draw[dashed] (\xmin,-\t) -- (\xmax,-\t);\
\draw[dashed] (\xmin,3*\t) -- (\xmax,3*\t);
\draw[dashed] (\xmin,5*\t) -- (\xmax,5*\t);
\draw[dashed] (\xmin,-3*\t) -- (\xmax,-3*\t);
\draw[dashed] (\xmin,-5*\t) -- (\xmax,-5*\t);
\draw[black,fill=black] (\y,\w) circle (.25ex) node[right]{$\rho_{0}(\yh)=\yh$};
\draw[black,fill=black] (\y,{2*\t*1+\w*(-1)^1}) circle (.25ex) node[right]{$\rho_{1}(\yh)$};
\draw[black,fill=black] (\y,{2*\t*2+\w*(-1)^2}) circle (.25ex) node[right]{$\rho_{2}(\yh)$};
\draw[black,fill=black] (\y,{2*\t*(-1)+\w*(-1)^(-1)}) circle (.25ex) node[right]{$\rho_{-1}(\yh)$};
\draw[black,fill=black] (\y,{2*\t*(-2)+\w*(-1)^(-2)}) circle (.25ex) node[right]{$\rho_{-2}(\yh)$};
\draw (-0.1,\t) node[left,fill=white] {$t$};
\draw (-0.1,-\t) node[left,fill=white] {$-t$};
\draw (-0.1,3*\t) node[left,fill=white] {$3t$};
\draw (-0.1,5*\t) node[left,fill=white] {$5t$};
\draw (-0.1,-3*\t) node[left,fill=white] {$-3t$};
\draw (-0.1,-5*\t) node[left,fill=white] {$-5t$};
\draw[-] (-0.15,\t) -- (0.15,\t);
\draw[-] (-0.15,-\t) -- (0.15,-\t);
\draw[-] (-0.15,3*\t) -- (0.15,3*\t);
\draw[-] (-0.15,-3*\t) -- (0.15,-3*\t);
\draw[-] (-0.15,5*\t) -- (0.15,5*\t);
\draw[-] (-0.15,-5*\t) -- (0.15,-5*\t);
\draw[->] (\xmin,0) -- (\xmax,0) node[right] {$x \in \Rn$};
\draw[->] (0,\zmin) -- (0,\zmax) node[above] {$z \in \R$};
\end{tikzpicture}
\end{center}
\caption{\label{fig:yhat} The points $\rho_j(\yh)$ are determined by repeated reflection across the strip boundaries at height $\pm t$.}
\end{figure}

The kernel is invariant under horizontal translations of the strip, and satisfies the Neumann boundary condition
\begin{equation} \label{eq:kernelneumann}
\left. \frac{\partial G_t}{\partial z} \right|_{z=\pm t} = 0
\end{equation}
on the upper and lower boundaries, by \cite[Proposition 4.1(b)]{CL26}.

\subsection*{Energy in Neumann strip, and equilibrium measures}
For $n \geq 2$, consider a nonempty compact set $K \subset \Rnp$. Supposing $t>0$ is large enough that $K \subset S(t)$, the strip energy of $K$ is defined to be
\[
E_K(t) = \min_\mu \int_K \! \int_K G_t(\xh,\yh) \, d\mu(\xh) d\mu(\yh) ,
\]
where the minimum is taken over all probability measures on $K$. (If $K$ is empty, define the energy to be $+\infty$.) Minimizing measures exist and are called equilibrium measures (see \cite[Lemma 4.1.3]{BHS19}). If the strip energy is finite and $K \subset \Rn$ then the equilibrium measure is unique, by \cite[Lemma 2.1]{CL26} applied with $q=n-1$. Whether or not it is unique, we denote an equilibrium measure by $\mu_t$.

In the limiting case $t \to \infty$, the strip fills all of $\Rnp$ and so it makes sense to define
\[
G_\infty(\xh,\yh)=\frac{1}{|\xh-\yh|^{n-1}}
\]
to be the $(n-1)$-Riesz kernel. With this definition,
\[
E_K(\infty)=V_{n-1}(K)
\]
is the $(n-1)$-Riesz energy.

The energy $E_K(t)$ is greater than $0$ when $n \geq 3$, and when $n=2$ one has $E_K(t)> -\infty$. Further, the energy is unchanged under horizontal translations of $K$, that is, under translations that preserve the $z$-coordinate.

\subsection*{Derivative formula with respect to strip thickness}

The following formula for the derivative of the strip energy $E_K(t)$, when $K$ is contained in the central hyperplane $\Rn$, is one of the keys to our main theorem \autoref{th:main}. It provides a formula for the rate of change of the energy with respect to variations of the boundary. Define a constant
\[
a_n = \frac{1}{(n-1)|\Sphn|} .
\]
\begin{theorem}[Strip energy derivative] \label{EKderivDirichlet}
Let $n \geq 2$. If $K \subset \Rn$ is compact and the energy $E_K(t)$ is finite for some $t>0$, then $E_K(t)$ is continuously differentiable for all $t>0$, with
\[
\big( t E_K(t) \big)^\prime = 2 a_n \int_{S(t)} \! \left( \frac{\partial U_t}{\partial z} \right)^{\!\! 2} \, dx dz
\]
where $U_t(\xh)= \int_K G_t(\xh,\yh) \, d\mu_t(\yh)$ is the unique equilibrium potential of $K$ on the strip $S(t)$.
\end{theorem}
\autoref{sec:dirichlet3} and \autoref{sec:dirichlet2} contain the proof.

\section{\texorpdfstring{\bf Proof of \autoref{EKderivDirichlet} for $n \geq 3$}{Proof of the strip energy derivative for n at least 3}}
\label{sec:dirichlet3}

The potential of a probability measure $\mu$ on a compact set $K \subset S(t)$ is
\[
U(\xh)= \int_K G_t(\xh,\yh) \, d\mu(\yh) , \qquad \xh \in \overline{S(t)} .
\]
Thanks to the kernel properties in the previous section, and in particular the Neumann condition \eqref{eq:kernelneumann}, the potential is smooth on $\overline{S(t)} \setminus K$ and satisfies the Neumann boundary condition
\[
\frac{\partial U}{\partial \nu} = 0
\]
on the upper and lower boundaries where $z = \pm t$. Further, the potential is superharmonic on $S(t)$ and is harmonic away from $K$, since the Newtonian kernel $\xh \mapsto 1/|\xh-\yh|^{n-1}$ is superharmonic on $\Rnp$ and is harmonic away from the singularity at $\yh$.

\subsection{Preliminary estimates on potentials}
Let
\[
c_m = \frac{\Gamma(\frac{1}{2}) \Gamma(\frac{m}{2})}{2\Gamma(\frac{m+1}{2})} , \qquad m>0 .
\]
Potentials in the $(n+1)$-dimensional strip decay like $r^{-(n-2)}$, which agrees with the Newtonian rate in the lower dimensional hyperplane $\Rn$.
\begin{proposition}[Potential decay at spatial infinity] \label{pr:potentialdecay} Suppose $K$ is a nonempty compact subset of the strip $S(t)\subset \Rnp$, where $n\geq 2$ and $t>0$ are fixed. Let $\mu$ be a probability measure on $K$ and write $U$ for the potential of $\mu$, as above.

Then as $r=|x| \to \infty$, the potential on $\overline{S(t)}$ decays according to
\[
U(x,z) =
\begin{cases}
\dfrac{c_{n-2}}{t} \dfrac{1}{r^{n-2}} + O \! \left( \dfrac{1}{r^{n-1}} \right) & \text{when $n \geq 3$,} \\[1em]
 \dfrac{1}{t} \log \dfrac{1}{r} + O \! \left( \dfrac{1}{r} \right) & \text{when $n=2$,}
\end{cases}
\]
and its gradient decays according to
\[
\nabla U(x,z) = - \frac{d_n}{tr^{n-1}} \, (x/r,0) + O\!\left( \frac{1}{r^n}\right) ,
\]
where $d_2=1$ and $d_n=(n-2)c_{n-2}$ when $n \geq 3$, and where the remainder estimates are uniform with respect to $z \in [-t,t]$.
\end{proposition}
\begin{proof}
The proposition follows by integrating with respect to $d\mu(\yh)$ the kernel estimates from \cite[Corollary 5.2]{CL26}, for the choice $q=n-1$. Here one uses the alternative form of the constant $qc_{q+1}$ that appears in that corollary, as stated after \cite[Proposition 5.1]{CL26}; or else directly use the functional equation for the gamma function to show $(n-1)c_n = (n-2)c_{n-2}$.
\end{proof}

\subsection{Dirichlet integrals for the strip energy} \label{sec:Dirichletintegrals}
Assume $n \geq 3$ throughout the remainder of \autoref{sec:dirichlet3}.

Fix $t>0$. Suppose $K \subset \Rnp$ is a compact subset of the strip $S(t)$ and $\mu$ is a probability measure on $K$. (We do not assume $\mu$ is an equilibrium measure for $K$.) Define the energy of $\mu$ with respect to the kernel $G_t$ to be
\[
E_\mu(t)=\int_K \! \int_K G_t(\xh,\yh) \, d\mu d\mu
\]
and assume $E_\mu(t)<\infty$. The immediate goal is to show in \autoref{StripDirichlet} that $E_\mu(t)$ equals the Dirichlet integral of the potential $U$ of $\mu$.

The potential can be written as
\[
U(\xh) =\int_K G_t(\xh,\yh) \, d\mu(\yh) = N(\xh)+H(\xh) , \qquad \xh=(x,z) \in \overline{S(t)} ,
\]
where we have decomposed the kernel \eqref{eq:rieszkernelq} into its Newtonian part and remainder as
\[
G_t(\xh,\yh)=\frac{1}{|\xh-\yh|^{n-1}}+H_t(\xh,\yh)
\]
and then decomposed the potential correspondingly in terms of
\[
N(\xh)=\int_K \frac{1}{|\xh-\yh|^{n-1}} \, d\mu(\yh) , \qquad H(\xh) = \int_K H_t(\xh,\yh) \, d\mu(\yh) .
\]
Thus $N$ is the Newtonian potential of $\mu$, which is smooth away from $K$ and decays according to
\[
N(\xh)=O(1/|x|^{n-1}) , \qquad \nabla N(\xh)=O(1/|x|^n) ,
\]
as $|x| \to \infty$. (In fact, $N$ decays on $\Rnp$ according to $N(\xh)=O(1/|\xh|^{n-1})$ and $\nabla N(\xh)=O(1/|\xh|^n)$.) The potential $U$ decays one order slower than $N$ as $|x| \to \infty$, since
\begin{equation} \label{Udecay}
U(\xh) = O(1/|x|^{n-2}) , \qquad \nabla U(\xh) = O(1/|x|^{n-1}) ,
\end{equation}
by \autoref{pr:potentialdecay}. Hence the remainder $H=U-N$, which is smooth and harmonic on the closure of $S(t)$, decays as
\[
H(\xh) = O(1/|x|^{n-2}) , \qquad \nabla H(\xh) = O(1/|x|^{n-1}) .
\]
In each case, the decay is uniform with respect to $z \in [-t,t]$.

Since $n \geq 3$, we deduce from these estimates that $\nabla H \in L^2(S(t))$, and that $\nabla U$ and $\nabla N$ are square integrable on $S(t)$ except perhaps on a neighborhood of $K$.

The next lemma states that the Newtonian energy of a measure equals the Dirichlet integral of its Newtonian potential, and so in fact $\nabla N$ is square integrable on all of $S(t)$. The lemma holds unchanged also when $n=2$. Recall that the Newtonian potential satisfies
\[
- a_n \Delta N = d\mu
\]
in the distributional sense, since the fundamental solution for the Laplacian in $\Rnp$ is $a_n/|\xh|^{n-1}$.
\begin{lemma}[Dirichlet integral for Newtonian energy] \label{NewtonianDirichlet}
The Newtonian potential $N$ is locally integrable and has weak gradient $\nabla N \in L^2(\Rnp)$, with
\[
a_n \int_\Rnp |\nabla N|^2 \, d\xh = \int_K \! \int_K \frac{1}{|\xh-\yh|^{n-1}} \, d\mu(\xh) d\mu(\yh) < \infty .
\]
\end{lemma}
\begin{proof}
The right side equals $\int_\Rnp N(\xh) \, d\mu(\xh)$. Formally, one may substitute the distributional Laplacian $d\mu=-a_n \Delta N$ and apply Green's theorem to obtain the left side, using that the Newtonian potential $N$ and its gradient vanish sufficiently fast at infinity to deduce that no boundary terms survive.

A rigorous proof is in Landkof \cite{L72}, as we now detail. Local integrability of $N$ on $\Rnp$ is clear since the singularity in the kernel is only of the form $|\xh|^{-(n-1)}$. Weak differentiability follows from \cite[Lemma 1.6]{L72}. That the weak gradient $\nabla N$ is square integrable and satisfies the equation in \autoref{NewtonianDirichlet} is in \cite[Theorem 1.20]{L72}.

\emph{Note.} Landkof's normalizations differ from ours. His Newtonian kernel and hence energy and potential equal $A(n+1,2)$ times our corresponding quantities, where $A(n+1,2)=\pi^{2-(n+1)/2} \Gamma(\frac{n-1}{2})$; see \cite[{\S}1.1]{L72}. The formula in \autoref{NewtonianDirichlet} is equivalent to his Theorem 1.20 because $a_n=A(n+1,2)/4\pi^2$.
\end{proof}

The next proposition and corollary derive a Dirichlet integral representation for the energy $E_\mu(t)$.
\begin{proposition}[Dirichlet integral for strip energy when $n \geq 3$] \label{StripDirichlet}
The potential $U(\xh)= \int_K G_t(\xh,\yh) \, d\mu(\yh)$ has weak derivatives in $L^2(S(t))$, with
\[
a_n \int_{S(t)} |\nabla U|^2 \, d\xh = \int_K \! \int_K G_t(\xh,\yh) \, d\mu(\xh) d\mu(\yh) = E_\mu(t) < \infty.
\]
\end{proposition}
\begin{proof}
The formal justification is the same as for the Newtonian case in \autoref{NewtonianDirichlet}, using additionally that $U$ satisfies the Neumann condition $\partial U/\partial \nu=0$ on the strip boundary $\partial S(t)$.

To justify the proposition rigorously, rather than adapting the full Newtonian proof from Landkof we utilize \autoref{NewtonianDirichlet} for the Newtonian energy and handle the other parts of the energy via Green's formula.

The energy can be written as the Newtonian part plus the rest:
\begin{align}
E_\mu(t)
& = \int_K \! \int_K \frac{1}{|\xh-\yh|^{n-1}} \, d\mu(\xh) d\mu(\yh) + \int_K \! \int_K H_t(\xh,\yh) \, d\mu(\xh) d\mu(\yh) \notag \\
& = a_n \int_\Rnp |\nabla N|^2 \, d\xh + \int_K H(\xh) \, d\mu(\xh) \label{eq:twoterms}
\end{align}
by \autoref{NewtonianDirichlet} for the Newtonian energy. The first term in \eqref{eq:twoterms} decomposes further into a sum of Dirichlet integrals inside and outside the strip. Applying Green's theorem to the portion outside the strip yields that
\begin{align*}
\int_\Rnp |\nabla N|^2 \, d\xh
& = \int_{S(t)} |\nabla N|^2 \, d\xh - \int_{\partial S(t)} N \frac{\partial N}{\partial \nu} \, dx
\end{align*}
since $N$ is harmonic on $\Rnp \setminus S(t)$ and decays fast enough at infinity; here $\nu$ is the outward normal from the strip.

For the second term in \eqref{eq:twoterms}, observe that $d\mu=-a_n \Delta U$ distributionally in the strip $S(t)$, since $U=N+H$ and $H$ is harmonic in the strip. Since $\nabla N$ and $\nabla H$ are square integrable on $S(t)$, Green's formula yields that
\begin{align*}
\int_K H(\xh) \, d\mu(\xh)
& = a_n \int_{S(t)} \nabla U \cdot \nabla H \, d\xh
\end{align*}
where the boundary terms vanish because $\partial U/\partial \nu = 0$ on $\partial S(t)$ while $\nabla U$ and $H$ decay like $O(1/|x|^{n-1})$ and $O(1/|x|^{n-2})$ as $|x| \to \infty$, respectively. In more detail, the integral of $H \frac{\partial U}{\partial \nu}$ with respect to $x$ over a sphere of radius $R$ in $\Rn$ is bounded by $O(1/R^{n-2})$, and that expression tends to $0$ as $R \to \infty$ since $n \geq 3$.

Combining these formulas for the two terms in \eqref{eq:twoterms}, and noting that
\begin{equation} \label{identity}
|\nabla N|^2 + \nabla U \cdot \nabla H = |\nabla U|^2 - \nabla N \cdot \nabla H
\end{equation}
because $U=N+H$, we find
\[
E_\mu(t) = a_n \left( \int_{S(t)} |\nabla U|^2 \, d\xh - \int_{\partial S(t)} N \frac{\partial N}{\partial \nu} \, dx - \int_{S(t)} \nabla N \cdot \nabla H \, d\xh \right) .
\]
The Neumann boundary condition $\partial U/\partial \nu=0$ on $\partial S(t)$ means that $-\partial N/\partial \nu = \partial H/\partial \nu$. Hence by Green's theorem and the decay estimates on $N$ and $\partial H/\partial \nu$, the second and third terms on the right side of the last formula sum to $a_n \int_{S(t)} N \Delta H \, d\xh$, which vanishes by harmonicity of $H$. The proposition is now proved.
\end{proof}

As previously, we write $\mu_t$ for an equilibrium measure on $K$ with respect to the kernel $G_t$, so that the corresponding equilibrium potential is
\[
U_t(\xh) = \int_K G_t(\xh,\yh) \, d\mu_t(\yh) , \qquad \xh \in \overline{S(t)} .
\]
\autoref{StripDirichlet} applies in particular to $\mu_t$ and $U_t$, and thus expresses the strip energy $E_K(t)$ in terms of the Dirichlet integral of $U_t$, as stated in the next corollary. The reciprocal of the energy turns out to be a useful quantity to consider too.
\begin{corollary}[Dirichlet integral for reciprocal energy when $n \geq 3$] \label{StripDirichletReciprocal}
If $K \subset \Rnp$ is a compact subset of the strip $S(t)$ and has finite energy $E_K(t)$, then
\begin{align*}
E_K(t) & = a_n \int_{S(t)} |\nabla U_t|^2 \, d\xh , \\
\frac{1}{E_K(t)} & = a_n \int_{S(t)} |\nabla u_t|^2 \, d\xh ,
\end{align*}
where $u_t=U_t/E_K(t)$ is a normalized equilibrium potential.
\end{corollary}

\subsection{\texorpdfstring{Variational characterization of the reciprocal energy when $n \geq 3$}{Variational characterization of the reciprocal energy when n is at least 3}}

The Dirichlet integral formulas in \autoref{StripDirichletReciprocal} motivate a variational characterization of the reciprocal energy.
\begin{proposition}[Variational characterization of the reciprocal energy when $n \geq 3$] \label{Characterization}
If $K \subset \Rnp$ is a compact subset of the strip $S(t), t>0$, and has finite energy $E_K(t)$, then
\begin{equation} \label{eq:variational}
\frac{1}{E_K(t)} = a_n \min_v \int_{S(t)} |\nabla v|^2 \, d\xh
\end{equation}
where the minimum is taken over the class of functions $v$ such that:
\begin{itemize}
\item $v$ is measurable and locally bounded on $S(t)$,
\item $v$ has weak gradient in $L^2(S(t))$,
\item $v(\xh)=O(1/|x|^{n-2})$ and $\nabla v(\xh) = O(1/|x|^{n-1})$ as $|x| \to \infty$, uniformly with respect to $z \in [-t,t]$ where $\xh=(x,z) \in S(t)$,
\item $v \geq 1$ on $K$ except possibly on a subset $Z \subset K$ that is a countable union of compact sets each having $(n-1)$-capacity zero,
\item $v$ is continuous at each point in $\overline{S(t)} \setminus Z$.
\end{itemize}
The minimum in \eqref{eq:variational} is achieved when $v$ is a normalized equilibrium potential $u_t=U_t/E_K(t)$.
\end{proposition}
Such variational characterizations (essentially the Dirichlet principle) are well known in the Newtonian case, and related capacity problems. The proof below follows the standard lines, while taking care with the decay at infinity and the Neumann conditions on the strip boundary.
\begin{proof}[Proof of \autoref{Characterization}]
When $v$ is taken to be a normalized equilibrium potential $u_t=U_t/E_K(t)$, equality holds in the variational formula \eqref{eq:variational} by \autoref{StripDirichletReciprocal}. Observe that $u_t$ satisfies all the hypotheses of the proposition, by \autoref{pr:potentialdecay}, \autoref{StripDirichlet} and \autoref{th:upperu}, which give that $u_t = 1$ on $K$ except on a countable union of compact subsets each having capacity zero. Hence ``$\geq$'' holds in \eqref{eq:variational}.

Next we prove ``$\leq$''. For notational simplicity, from now on write $U_t$ as $U$, $u_t$ as $u$, and $\mu_t$ as $\mu$. The task is to show that the upper bound
\begin{equation} \label{eq:vineq}
\frac{1}{E_K(t)} \leq a_n \int_{S(t)} |\nabla v|^2 \, d\xh
\end{equation}
holds whenever $v$ possesses the properties specified in the proposition.

Let
\[
w=v-u .
\]
We will show that the gradients $\nabla u$ and $\nabla w$ have nonnegative $L^2$-inner product:
\begin{equation} \label{orthogonal}
\int_{S(t)} \nabla u \cdot \nabla w \, d\xh \geq 0 .
\end{equation}
Assuming this formula for now, we deduce that
\begin{align*}
\frac{1}{a_n} \frac{1}{E_K(t)}
& \leq \int_{S(t)} |\nabla u|^2 \, d\xh + 2 \int_{S(t)} \nabla u \cdot \nabla w \, d\xh \qquad \text{by \autoref{StripDirichletReciprocal} and \eqref{orthogonal}} \\
& = \int_{S(t)} |\nabla (u+w)|^2 \, d\xh - \int_{S(t)} |\nabla w|^2 \, d\xh \\
& \leq \int_{S(t)} |\nabla v|^2 \, d\xh
\end{align*}
since $u+w=v$. That proves \eqref{eq:vineq} and hence the proposition.

Thus the task is to show nonnegativity of the inner product in \eqref{orthogonal}. We start by observing that $u$ and $w$ decay as $|x| \to \infty$ within the strip, since the decay estimate \eqref{Udecay} applies to $u$ while $v$ decays similarly by hypothesis. Further, each function has square-integrable weak derivatives on $S(t)$. Let $0<\e<t$, extend $w$ locally across the boundaries of the strip at $z=\pm t$ by even reflection, and write $w^{(\e)}$ for the $\e$-mollification of this extended $w$ with respect to a smooth bump function supported in the unit ball of $\Rnp$. Since $a_n \Delta U = -d\mu$ in the distributional sense and $w^{(\e)}$ is smooth on $\overline{S(t)}$, Green's theorem implies that
\[
a_n \int_{S(t)} \nabla U \cdot \nabla w^{(\e)} \, d\xh
= \int_K w^{(\e)} \, d\mu + a_n \int_{\partial S(t)} \frac{\partial U}{\partial \nu} \, w^{(\e)} \, dS
\]
where the boundary term at infinity is absent since $\partial U/\partial \nu = O(1/|x|^{n-1})$ and $w^{(\e)}=O(1/|x|^{n-2})$ by the decay estimates above, as $|x| \to \infty$. In the second term on the right, the Neumann boundary condition ensures that $\partial U/\partial \nu=0$ on $\partial S(t)$, and so that second term equals $0$. On the left side, $\nabla w^{(\e)} \to \nabla w$ in $L^2(S(t))$ as $\e \to 0$.

Thus to prove \eqref{orthogonal}, we want to show that the first term on the right side converges to a nonnegative limit: $\lim_{\e \to 0} \int_K w^{(\e)} \, d\mu \geq 0$. Note that $u$ and $v$ and hence also $w$ are locally bounded. Hence it suffices to show pointwise convergence, namely that $\lim_{\e \to 0} w^{(\e)} \geq 0$ on $K$ except for a set of $\mu$-measure zero.

Let $Z_u$ and $Z_v$ be subsets of $K$ that can be written as countable unions of compact sets having $(n-1)$-capacity zero such that $u$ is continuous at each point of $\overline{S(t)} \setminus Z_u$ with $u=1$ on $K \setminus Z_u$, and $v$ is continuous at each point of $\overline{S(t)} \setminus Z_v$ with $v \geq 1$ on $K \setminus Z_v$. Let $Z=Z_u \cup Z_v$, so that $w$ is continuous at each point in $\overline{S(t)} \setminus Z$ and $Z$ is a countable union of compact sets having $(n-1)$-capacity zero. The mollified function converges pointwise with $w^{(\e)}(\xh) \to w(\xh)$ for each $\xh \in K \setminus Z$. Thus $\lim_{\e \to 0} w^{(\e)} = w(\xh) \geq 0$, since $v(\xh) \geq 1$ and $u(\xh)=1$. To finish the proof of this step, note that $\mu(Z)=0$ because every compact set of capacity zero has $\mu$-measure zero.
\end{proof}

Continue to assume $n \geq 3$. We aim to invoke the variational characterization in \autoref{Characterization} to obtain an upper bound on the reciprocal energy and then differentiate the reciprocal energy with respect to the strip thickness parameter $t$.

First we construct a good trial function for the variational characterization. Fix $\tau>0$ and consider a normalized equilibrium potential $u_\tau=U_\tau/E_K(\tau)$ on $S(\tau)$. Rescale it in the vertical direction and construct a trial function $v_t$ on $S(t)$ by defining
\begin{equation} \label{defvt}
v_t(x,z) = u_\tau(x,\tau z/t) , \qquad x \in \Rn , \quad z \in [-t,t] ,
\end{equation}
when $t>0$. In the special case $t=\tau$, the function is preserved since $v_\tau=u_\tau$.
\begin{corollary}[Variational upper bound when $n \geq 3$, for sets in $\Rn$] \label{StripDirichletReciprocalUpper}
Let $\tau>0$. If $K \subset \Rn$ is compact with finite energy $E_K(\tau)$ then
\[
\frac{1}{E_K(t)} \leq a_n \int_{S(t)} |\nabla v_t|^2 \, d\xh
\]
for all $ t>0$, with equality at $t=\tau$.
\end{corollary}
\begin{proof}
The corollary will follow immediately from \autoref{Characterization} once we confirm that the hypotheses for $v$ in that result are satisfied by $v_t$. The $z$-rescaling in the definition of $v_t$ does not materially affect those conditions, since the assumption $K \subset \Rn$ ensures that $K$ remains fixed under the rescaling and so $v_t=u_\tau$ on $K$.

Thus it is enough to check the corresponding properties for the normalized equilibrium potential $u_\tau$ on $S(\tau)$, which we did already in the proof of \autoref{Characterization}.
\end{proof}

\subsection{\texorpdfstring{Proof of \autoref{EKderivDirichlet} when $n \geq 3$}{Proof of the strip energy derivative when n is at least 3}}
First, $E_K(t)$ is continuously differentiable by \cite[Theorem 3.4]{CL26} with $q=n-1$. The proof of that theorem uses uniqueness of the equilibrium measure \cite[Lemma 2.1]{CL26}, relying on our hypothesis $K \subset \Rn$. Note that although the theorem gives a formula for $E^\prime(t)$ in terms of the $t$-derivative $\partial G_t/\partial t$ of the kernel, we have not found a way to use it productively and so instead proceed as follows.

Fix $\tau>0$. Using difference quotients,
\begin{align}
\left( \frac{1}{E_K(t)} \right)_{t=\tau}^{\!\! \prime}
& = \lim_{t \searrow \tau} \frac{E_K(t)^{-1} - E_K(\tau)^{-1}}{t-\tau} \label{diffquotient} \\
& \leq a_n \left. \frac{d\ }{dt} \, \int_{S(t)} |\nabla v_t|^2 \, d\xh \, \right|_{t=\tau} \label{diffquotuppper}
\end{align}
by \autoref{StripDirichletReciprocalUpper}, which again uses the hypothesis $K \subset \Rn$. The derivative in \eqref{diffquotuppper} does indeed exist, as we proceed to demonstrate.

In view of the definition of $v_t$ in \eqref{defvt}, the Dirichlet integral of $v_t$ decomposes into horizontal and vertical parts as follows:
\begin{align}
& \int_{S(t)} |\nabla v_t|^2 \, d\xh \notag \\
& = \int_\Rn \int_{-t}^t \left\{ |(\nabla_{\! x} u_\tau)(x,\tau z/t)|^2 + (\tau/t)^2 |(\partial_z u_\tau)(x,\tau z/t)|^2 \right\} dz dx \notag \\
& = (t/\tau) \int_\Rn \int_{-\tau}^\tau \left\{ |\nabla_{\! x} u_\tau(x,z)|^2 + (\tau/t)^2 |\partial_z u_\tau(x,z)|^2 \right\} dz dx \qquad \text{by $z \mapsto (t/\tau) z$} \notag \\
& = (t/\tau) \int_{S(\tau)} \left\{ |\nabla_{\! x} u_\tau|^2 + (\tau/t)^2 |\partial_z u_\tau|^2 \right\} d\xh . \label{changesofvariable}
\end{align}
Thus the derivative in \eqref{diffquotuppper} evaluates to give that
\begin{align}
\left( \frac{1}{E_K(t)} \right)_{t=\tau}^{\!\! \prime}
& \leq a_n \, \frac{1}{\tau} \int_{S(\tau)} |\nabla u_\tau|^2 \, d\xh - a_n \, \frac{2}{\tau} \int_{S(\tau)} |\partial_z u_\tau|^2 \, d\xh \notag \\
& = \frac{1}{\tau E_K(\tau)} - \frac{2a_n}{\tau} \int_{S(\tau)} |\partial_z u_\tau|^2 \, d\xh \label{upperrightderiv}
\end{align}
by \autoref{StripDirichletReciprocal}. Multiplying through the inequality by $-\tau E_K(\tau)^2$ implies that
\[
\tau E_K^{\, \prime}(\tau) \geq - E_K(\tau) + 2a_n \int_{S(\tau)} |\partial_z U_\tau|^2 \, d\xh
\]
since $u_\tau = U_\tau/E_K(\tau)$. Thus we have proved the ``$\geq$'' direction of \autoref{EKderivDirichlet}.

To obtain the ``$\leq$'' direction, consider the difference quotient in \eqref{diffquotient} but now with $t \nearrow \tau$, so that $t-\tau$ is negative and the inequality in \eqref{diffquotuppper} reverses direction. The inequalities in the rest of the argument reverse also.

\subsubsection*{Remark} This proof relies on our knowing that the energy derivative $E_K^{\, \prime}(\tau)$ exists, for which fact we used \cite[Theorem 3.4]{CL26}. Without that differentiability, the proof would show only that the upper right-derivative of $(E_K)^{-1}$ is less than or equal to \eqref{upperrightderiv} and that the lower left-derivative is greater than or equal to the same expression; we would have no information on the lower right-derivative or the upper left-derivative, and so could not conclude the proposition.

\section{\texorpdfstring{\bf Proof of \autoref{EKderivDirichlet} for $n=2$}{Proof of the strip energy derivative for n=2}}
\label{sec:dirichlet2}

This section proves \autoref{EKderivDirichlet} for $n=2$.

\subsection{Dirichlet integrals for the strip energy} Fix $t>0$, suppose $K \subset \R^3$ is a compact subset of the strip $S(t) = \R^2 \times (-t,t)$, with $K$ not necessarily lying in the central plane $\R^2$, and that $\mu$ is a probability measure on $K$ with finite energy $E_\mu(t)=\int_K \! \int_K G_t(\xh,\yh) \, d\mu d\mu<\infty$. As in \autoref{sec:Dirichletintegrals}, the potential of $\mu$ can be written
\[
U(\xh) = N(\xh)+H(\xh) , \qquad \xh=(x,z) \in \overline{S(t)} ,
\]
where (remembering that $n=2$) one has
\[
N(\xh)=\int_K \frac{1}{|\xh-\yh|} \, d\mu(\yh) , \qquad H(\xh) = \int_K H_t(\xh,\yh) \, d\mu(\yh) .
\]
Note $N$ is the Newtonian potential of $\mu$ in $\R^3$, which is smooth away from $K$ and decays as $|x| \to \infty$ in the strip according to
\[
N(\xh)=O(1/|x|) , \qquad \nabla N(x)=O(1/|x|^2) ,
\]
uniformly with respect to $z \in [-t,t]$. (In fact, $N$ decays on $\R^3$ according to $N(\xh)=O(1/|\xh|)$ and $\nabla N(\xh)=O(1/|\xh|^2)$.) Meanwhile, $H$ is harmonic and smooth on the closure of $S(t)$. For $U$, taking $n=2$ in \autoref{pr:potentialdecay} gives that
\begin{equation} \label{Udecay2}
\nabla U(\xh) = O(1/|x|)
\end{equation}
as $|x| \to \infty$. This decay rate is not quite enough for square integrability of the gradient on $S(t)$, and for that reason the Dirichlet integral formula in the next result looks somewhat different from the one in \autoref{StripDirichlet}. It involves subtracting a logarithmic term and taking the limit over a family of cylinders that expand to fill the strip.
\begin{proposition}[Dirichlet integral for strip energy when $n=2$] \label{StripDirichlet2}
The potential $U(\xh)= \int_K G_t(\xh,\yh) \, d\mu(\yh)$ when $n=2$ has weak derivatives in $L^2_{loc}(S(t))$, with
\begin{align*}
\lim_{R \to \infty} & \left( a_2 \int_{S(R,t)} |\nabla U|^2 \, d\xh - \frac{1}{t} \log R \right) \\
& = \int_K \! \int_K G_t(\xh,\yh) \, d\mu(\xh) d\mu(\yh) = E_\mu(t) < \infty
\end{align*}
where $a_2=1/4\pi$ and $S(R,t) = \{ (x,z) \in \R^2 \times \R : |x|<R, |z| < t \}$ is a cylinder of height $2t$ and radius $R$.
\end{proposition}
\begin{proof}
Decompose the energy into the Newtonian part and the remainder as
\begin{align}
E_\mu(t)
& = \int_K \! \int_K \frac{1}{|\xh-\yh|} \, d\mu(\xh) d\mu(\yh) + \int_K \! \int_K H_t(\xh,\yh) \, d\mu(\xh) d\mu(\yh) \notag \\
& = a_2 \int_{\R^3} |\nabla N|^2 \, d\xh + \int_K H(\xh) \, d\mu(\xh) \label{eq:twoterms2}
\end{align}
by \autoref{NewtonianDirichlet}, which holds as noted there for $n=2$. The first term in \eqref{eq:twoterms2} further decomposes into the Dirichlet integrals inside and outside the cylinder. Choose $R$ large enough that $K \subset S(R,t)$. Applying Green's theorem outside the cylinder, where $N$ is harmonic, yields that
\begin{equation} \label{2dimnablaV}
\int_{\R^3} |\nabla N|^2 \, d\xh
= \int_{S(R,t)} |\nabla N|^2 \, d\xh - \int_{\partial S(R,t)} N \frac{\partial N}{\partial \nu} \, dS
\end{equation}
where $\nu$ is the outward normal from the cylinder.

To deal with the second term in \eqref{eq:twoterms2}, note that the Newtonian potential satisfies $- a_2 \Delta N = d\mu$ in the distributional sense on $\R^3$, where $a_2=1/4\pi$. Since $U=N+H$ and $H$ is harmonic in the strip, it follows that $-a_2 \Delta U=d\mu$ there. Because $\nabla N$ and $\nabla H$ are locally square integrable, Green's formula on the cylinder $S(R,t)$ yields that
\begin{equation} \label{2dimW}
\int_K H(\xh) \, d\mu(\xh)
= a_2 \left( \int_{S(R,t)} \nabla U \cdot \nabla H \, d\xh - \int_{\partial S(R,t)} H \frac{\partial U}{\partial \nu} \, dS \right) .
\end{equation}

Substituting \eqref{2dimnablaV} and \eqref{2dimW} into \eqref{eq:twoterms2} and then applying identity \eqref{identity}, we find
\begin{align*}
E_\mu(t)
& = a_2 \left( \int_{S(R,t)} |\nabla U|^2 \, d\xh - \int_{\partial S(R,t)} N \frac{\partial N}{\partial \nu} \, dS \right. \\
& \qquad \left. - \int_{S(R,t)} \nabla N \cdot \nabla H \, d\xh - \int_{\partial S(R,t)} H \frac{\partial U}{\partial \nu} \, dS \right) \\
& = a_2 \left( \int_{S(R,t)} |\nabla U|^2 \, d\xh - \int_{\partial S(R,t)} N \frac{\partial N}{\partial \nu} \, dS \right. \\
& \qquad \left. - \int_{\partial S(R,t)} N \frac{\partial H}{\partial \nu} \, dS - \int_{\partial S(R,t)} H \frac{\partial U}{\partial \nu} \, dS \right)
\end{align*}
by Green's theorem on the third term, since $\Delta H=0$ in the cylinder. Since $N+H=U$, the last formula simplifies to
\begin{align}
E_\mu(t)
& = a_2 \left( \int_{S(R,t)} |\nabla U|^2 \, d\xh - \int_{\partial S(R,t)} U \frac{\partial U}{\partial \nu} \, dS \right) \notag \\
& = a_2 \int_{S(R,t)} |\nabla U|^2 \, d\xh - \frac{1}{t} \log R + O \! \left( \frac{\log R}{R} \right) , \label{logRoverR}
\end{align}
where to arrive at \eqref{logRoverR} we used that $\partial U/\partial \nu=0$ on the top and bottom of the cylinder and that along the side $\{ |x|=R \}$, which has area $2\pi R \cdot 2t$, one has
\[
U = - \frac{1}{t} \log R + O(1/R), \qquad \frac{\partial U}{\partial r} = - \frac{1}{tR} + O(1/R^2) ,
\]
by \autoref{pr:potentialdecay} with $n=2$. Letting $R \to \infty$ in \eqref{logRoverR} completes the proof of \autoref{StripDirichlet2}.
\end{proof}

\subsection{\texorpdfstring{Proof of \autoref{EKderivDirichlet} when $n=2$}{Proof of the strip energy derivative when n=2}} The argument for $n=2$ that we develop in this section builds on the idea, used above for $n \geq 3$, of constructing a trial function for a variational characterization. The variational nature of the argument is obscured somewhat by the fact that the Dirichlet integral formula for the energy in \autoref{StripDirichlet2} involves a limiting process as $R \to \infty$. To circumvent this difficulty, we develop $t$-difference quotient estimates on the energy that persist even after letting $R \to \infty$.

\subsubsection*{Step 1 --- properties of the rescaled potential} Fix $\tau>0$. Let $U_t$ be the equilibrium potential for $E_K(t)$, for $t>0$. Define
\[
u_{R,t}(\xh) = \frac{t U_t(\xh) + \log R}{t E_K(t) + \log R} , \qquad \xh \in \overline{S(t)},
\]
whenever $R>0$ is large enough that the denominator is positive. Two facts will later be useful: that $u_{R,t}(\xh)=1$ if and only if $U_t(\xh) = E_K(t)$, and that at radius $|x|=R$ the function can be bounded according to
\begin{equation} \label{eq:littleulargeR}
\left. u_{R,t}(\xh) \right|_{r=R} = \frac{O(1/R)}{t E_K(t) + \log R}
\end{equation}
by the potential estimate in \autoref{pr:potentialdecay} with $n=2$, where the constant implicit in the $O(1/R)$ estimate is independent of $z \in [-t,t]$ but can depend on $t$. Importantly, in the proof that follows we will let $R \to \infty$ before any changes are made to the values of $t$ or $\tau$.

We rescale from $S(\tau)$ to $S(t)$ by defining
\[
v_{R,t}(x,z) = u_{R,\tau}(x,\tau z/t) , \qquad (x,z) \in \overline{S(t)} .
\]
Let
\[
w_{R,t}=v_{R,t}-u_{R,t} .
\]
We claim that
\begin{equation} \label{orthogonal2}
\int_{S(R,t)} \nabla u_{R,t} \cdot \nabla w_{R,t} \, d\xh = O(1/R(\log R)^2) .
\end{equation}

To begin proving \eqref{orthogonal2}, let $0<\e<t$, extend $w_{R,t}$ by even reflection across the boundaries of the strip at $z=\pm t$, and write $w_{R,t}^{(\e)}$ for the $\e$-mollification of $w_{R,t}$ with respect to a smooth bump function supported in the unit ball of $\R^3$. Since $- a_2 \Delta U_t = d\mu_t$ in the distributional sense and $w_{R,t}^{(\e)}$ is smooth on a neighborhood of $\overline{S(t)}$, Green's theorem implies that
\begin{equation} \label{crossterm}
a_2 \int_{S(R,t)} \nabla U_t \cdot \nabla w_{R,t}^{(\e)} \, d\xh
= \int_K w_{R,t}^{(\e)} \, d\mu_t + a_2 \int_{\partial S(R,t)} \frac{\partial U_t}{\partial \nu} \, w_{R,t}^{(\e)} \, dS .
\end{equation}
On the left side, $\nabla w_{R,t}^{(\e)} \to \nabla w_{R,t}$ in $L^2(S(R,t))$ as $\e \to 0$. On the right side, we will show the first integral tends to zero as $\e \to 0$.

Observe that $w_{R,t}^{(\e)}$ is bounded on $K$, independently of $\e$, by recalling the definitions of $u_{R,t}$ and $u_{R,\tau}$ and using the local boundedness of $U_t$ and $U_\tau$ from \autoref{th:upperu}. Thus to prove that $\int_K w_{R,t}^{(\e)} \, d\mu_t$ tends to zero, it is enough to show $\lim_{\e \to 0} w_{R,t}^{(\e)} = 0$ pointwise on $K$ except for a set of $\mu_t$-measure zero. By \autoref{th:upperu}, subsets $Z_\tau$ and $Z_t$ of $K$ exist that can be written as countable unions of compact sets having $(n-1)$-capacity zero and for which $U_\tau$ is continuous at each point of $\overline{S(\tau)} \setminus Z_\tau$ and $U_\tau=E_K(\tau)$ on $K \setminus Z_\tau$, and $U_t$ is continuous at each point of $\overline{S(t)} \setminus Z_t$ and $U_t = E_K(t)$ on $K \setminus Z_t$. Let $Z=Z_\tau \cup Z_t$, so that $Z$ is a countable union of compact sets having $(n-1)$-capacity zero. From the definitions, together with the fact that $K \subset \Rn$ is invariant under the rescaling of the $z$-variable, we see that $u_{R,t}$ and $v_{R,t}$ are continuous at each point of $\overline{S(t)} \setminus Z$ and they both equal $1$ on $K \setminus Z$. Thus $w_{R,t}$ is continuous at each point in $\overline{S(t)} \setminus Z$ and equals $0$ on $K \setminus Z$. By the continuity, the mollified function converges pointwise with $w_{R,t}^{(\e)}(\xh) \to w_{R,t}(\xh) = 0$ as $\e \to 0$, for each $\xh \in K \setminus Z$. The exceptional points satisfy $\mu_t(Z)=0$, because compact sets of capacity zero have $\mu_t$-measure zero. Thus we have proved for \eqref{crossterm} that $\int_K w_{R,t}^{(\e)} \, d\mu_t \to 0$ as $\e \to 0$.

Next we show the second integral on the right of \eqref{crossterm} is bounded by $O(1/R\log R)$, independently of $\e$. The Neumann boundary condition for $U_t$ ensures that $\partial U_t/\partial \nu=0$ on the top and bottom of the cylinder $S(R,t)$. On the sides where $r=R$, the normal derivative is simply the radial derivative, and so $\partial U_t/\partial \nu = O(1/R)$ by \eqref{Udecay2}. Further, $w_{R,t}(\xh)|_{r=R}=O(1/R \log R)$ by \eqref{eq:littleulargeR} applied to $u_{R,t}$ and $v_{R,t}$, and so for the mollified function we get $w_{R,t}^{(\e)}(\xh)=O(1/R\log R)$ with bounds that are independent of $\e$. The sides of the cylinder have area $2\pi R \cdot 2t$, and thus we conclude that the second integral on the right of \eqref{crossterm} is bounded by $O(1/R\log R)$. Now letting $\e \to 0$ in \eqref{crossterm} yields that
\[
a_2 \int_{S(R,t)} \nabla U_t \cdot \nabla w_{R,t} \, d\xh = O(1/R\log R) .
\]
After recalling the definition of $u_{R,t}$ in terms of $U_t$, we obtain \eqref{orthogonal2}.

\subsubsection*{Step 2 --- the variational-type inequality and its consequences} We have
\begin{align*}
& \int_{S(R,t)} |\nabla u_{R,t}|^2 \, d\xh \\
& = \int_{S(R,t)} |\nabla u_{R,t}|^2 \, d\xh + 2 \int_{S(R,t)} \nabla u_{R,t} \cdot \nabla w_{R,t} \, d\xh + O(1/R(\log R)^2) \quad \text{by \eqref{orthogonal2}} \\
& = \int_{S(R,t)} |\nabla (u_{R,t}+w_{R,t})|^2 \, d\xh - \int_{S(R,t)} |\nabla w_{R,t}|^2 \, d\xh +O(1/R(\log R)^2) \\
& \leq \int_{S(R,t)} |\nabla v_{R,t}|^2 \, d\xh + O(1/R(\log R)^2)
\end{align*}
since $u_{R,t}+w_{R,t}=v_{R,t}$. Expressing this variational-type inequality back in terms of $U_t$ and $U_\tau$, it implies
\begin{align}
& t^2 \int_{S(R,t)} |\nabla U_t(x,z)|^2 \, dx dz \notag \\
& \leq \left( \frac{t E_K(t) + \log R}{\tau E_K(\tau) + \log R} \right)^{\!\! 2} \tau^2 \int_{S(R,t)} |\nabla (U_\tau(x,\tau z/t))|^2 \, dx dz + O(1/R) . \label{2dimmidpoint}
\end{align}

For the pre-factor, we calculate as $R \to \infty$ that
\[
\left( \frac{t E_K(t) + \log R}{\tau E_K(\tau) + \log R} \right)^{\!\! 2} = 1 + 2 \frac{tE_K(t)-\tau E_K(\tau)}{\log R} + O(1/(\log R)^2) .
\]
Also, by using the same change of variable as in the derivation of \eqref{changesofvariable}, one finds that
\begin{align*}
& \int_{S(R,t)} |\nabla (U_\tau(x,\tau z/t))|^2 \, dx dz \\
& = \frac{t}{\tau} \left( \int_{S(R,\tau)} |\nabla U_\tau|^2 \, dxdz + \left( \left( \frac{\tau}{t} \right)^2 - 1 \right) \int_{S(R,\tau)} |\partial_z U_\tau|^2 \, dxdz \right) .
\end{align*}
Substituting the last two formulas into \eqref{2dimmidpoint}, and multiplying through by $a_2/t$ gives
\begin{align}
& a_2 t \int_{S(R,t)} |\nabla U_t|^2 \, dx dz \notag \\
& \leq \left( 1 + 2 \frac{tE_K(t)-\tau E_K(\tau)}{\log R} + O(1/(\log R)^2) \right) \label{eq:middleline} \\
& \quad \times \left( a_2 \tau \int_{S(R,\tau)} |\nabla U_\tau|^2 \, dxdz + a_2 \tau \left( \left( \frac{\tau}{t} \right)^2 - 1 \right) \int_{S(R,\tau)} |\partial_z U_\tau|^2 \, dxdz \right) + O(1/R) . \notag
\end{align}
Note that
\begin{equation} \label{eq:usefulestimates}
a_2 \tau \int_{S(R,\tau)} |\nabla U_\tau|^2 \, dxdz = \log R + O(1) , \qquad \int_{S(R,\tau)} |\partial_z U_\tau|^2 \, dxdz = O(1) ,
\end{equation}
from \autoref{StripDirichlet2} for the first equation, and for the second equation using that $\partial_z U_\tau$ decays like $O(1/r^2)$ by \autoref{pr:potentialdecay} with $n=2$. Now expand the right side of inequality \eqref{eq:middleline}; and for the terms that contain $1/\log R$ or $1/(\log R)^2$, substitute as in \eqref{eq:usefulestimates}; and then rearrange the resulting formula. One obtains that
\begin{align*}
& a_2 \tau \frac{(t^2-\tau^2)}{t^2} \int_{S(R,\tau)} |\partial_z U_\tau|^2 \, dxdz + O(1/\log R) \\
& \leq - a_2 t \int_{S(R,t)} |\nabla U_t|^2 \, dx dz + a_2 \tau \int_{S(R,\tau)} |\nabla U_\tau|^2 \, dxdz + 2 \left( tE_K(t)-\tau E_K(\tau) \right) .
\end{align*}
Letting $R \to \infty$ and applying \autoref{StripDirichlet2} on the right side for $t$ and also for $\tau$, we find
\[
a_2 (t-\tau) \frac{\tau(t+\tau)}{t^2} \int_{S(\tau)} |\partial_z U_\tau|^2 \, d\xh \leq tE_K(t)-\tau E_K(\tau) .
\]

\subsubsection*{Step 3 --- limit of difference quotients} Supposing $t>\tau$, we may divide the preceding formula by $(t-\tau)$ and let $t \searrow \tau$ to show that
\[
2 a_2 \int_{S(\tau)} |\partial_z U_\tau|^2 \, d\xh \leq \left. (tE_K(t))^\prime \right|_{t=\tau} .
\]
The reverse inequality follows by supposing $t<\tau$, dividing by $(t-\tau)$, and letting $t \nearrow \tau$. Hence equality holds, and so the proof of \autoref{EKderivDirichlet} is complete when $n=2$.

\subsubsection*{Remark} Just as in the higher dimensional case treated in the previous section, this proof for $n=2$ relies on our knowing by \cite[Theorem 3.4]{CL26} that $E_K(t)$ is differentiable.

\section{\bf Proof of the main result: \autoref{th:main}}
\label{sec:stripmin}

Fix $n \geq 2$. The goal in \autoref{conj:riesz} is to prove the ratio $\capnone(K)/\capntwo(K)$ is maximal when the compact set $K \subset \Rn$ is a closed ball. We assume $\capntwo(K)>0$ so that the ratio is well defined and $\capnone(K)>0$ so that it is not zero.

Choose $B$ to be a closed ball in $\Rn$ centered at the origin having $\capnone(B)=\capnone(K)$, or equivalently $V_{n-1}(K)=V_{n-1}(B)$. In the notation of \autoref{sec:stripenergy}, this means $E_K(\infty)=E_B(\infty)$. We will show $B$ has larger energy in each strip $S(t)$:
\begin{equation}\label{eq:t-energy-ineq}
E_K(t) \leq E_B(t), \qquad t \in (0,\infty) .
\end{equation}
Multiplying this inequality by $t$ and letting $t \to 0$, one deduces when $n\geq3$ that
\[
c_{n-2}V_{n-2}(K)
\leq \liminf_{t\searrow0}tE_K(t)
\leq \lim_{t\searrow0}tE_B(t)
=c_{n-2}V_{n-2}(B) ,
\]
by applying \cite[Theorem 3.3]{CL26} part (a) to $K$ and part (c) to $B$, with $q=n-1$. Hence
\[
V_{n-2}(K) \leq V_{n-2}(B) , \qquad n \geq 3 .
\]
For $n=2$, the corresponding argument gives
\[
V_{log}(K) \leq V_{log}(B) , \qquad n=2 .
\]
Thus for each $n$, we conclude $\capntwo(K) \geq \capntwo(B)$ and so the ball $B$ maximizes the ratio $\capnone(K)/\capntwo(K)$, completing the proof.

Thus the task is to prove \eqref{eq:t-energy-ineq}. First we dispose of an equality case. Suppose $K=\widetilde{B} \cup Z$ is the union of a closed ball $\widetilde{B} \subset \Rn$ and a set $Z \subset \Rn$ of inner $(n-1)$-capacity zero. Then
\[
E_{\widetilde{B}}(t)=E_{\widetilde{B} \cup Z}(t)=E_K(t) , \qquad t \in (0,\infty] ,
\]
by \autoref{le:Z}. At $t=\infty$ this equation says that $V_{n-1}(\widetilde{B})=V_{n-1}(K)$, and so the ball $\widetilde{B}$ necessarily has the same radius as $B$. Hence $E_B(t)=E_K(t)$ for all $t \in (0,\infty]$ and so \eqref{eq:t-energy-ineq} holds with equality for all $t$.

Suppose from now on that $K$ does not equal the union of a ball in $\Rn$ with a set of inner $(n-1)$-capacity zero. Let $x_K=\int_K x \,d\mu_{\infty}(x)$ be the centroid of the $(n-1)$-equilibrium measure $\mu_\infty$ on $K\subset \Rn$. A moment comparison theorem \cite[Theorem 1.2]{CL24} yields that the second moment of the $(n-1)$-equilibrium measure on the set $K-x_K$ is strictly greater than the second moment of the $(n-1)$-equilibrium measure on the ball $B\subset \Rn$, so that
\[
\int_K |x-x_K|^2 \, d\mu_\infty > \int_B |x|^2 \, d\nu_\infty
\]
where $\nu_\infty$ is the $(n-1)$-equilibrium measure on $B$. (Strictness holds in this comparison inequality thanks to the assumption that $K$ is not the union of a ball with a set of inner $(n-1)$-capacity zero.) The centroid for the ball is the origin, $x_B=0$, by symmetry. Applying this strict moment inequality to the strip energy asymptotic for $t \to \infty$ in \cite[Corollary 3.2]{CL26} (as usual taking $q=n-1$ there), we deduce from the assumption $E_K(\infty)=E_B(\infty)$ that
\begin{equation}\label{eq:t-energy-ineq2}
E_K(t) < E_B(t) \qquad \text{for all large $t$.}
\end{equation}

If inequality \eqref{eq:t-energy-ineq2} holds for all $t>0$ then \eqref{eq:t-energy-ineq} is proved and we are done. Suppose instead that \eqref{eq:t-energy-ineq2} fails for some $t$. We will deduce a contradiction. Writing $t_0$ for the largest $t$ at which it fails, we have $E_K(t_0)=E_B(t_0)$ and $E_K(t) < E_B(t)$ for all $t>t_0$. Hence $E_K^{{\,}\prime}(t_0)\leq E_B^{{\,}\prime}(t_0)$. Recall here that $E_K(t)$ is continuous and in fact continuously differentiable, by \cite[Theorem 3.4]{CL26}.

To obtain the desired contradiction with the help of \autoref{conj:neumann} we will show, after changing $t_0$ to $t$ for notational simplicity, that
\[
E_K(t)=E_B(t) \quad \Longrightarrow \quad E_K^{{\,}\prime}(t) > E_B^{{\,}\prime}(t) .
\]

Write $U_t$ and $V_t$ for the equilibrium potentials of the sets $K$ and $B$ for the energies $E_K(t)$ and $E_B(t)$, respectively, on the strip $S(t)$. We rescale them to functions $U$ and $V$ on the strip $S=S(1)$ by defining
\[
U(\xh) = \frac{t^{n-1}}{c_{n-2}} \, U_t(t \xh) , \qquad V(\xh) = \frac{t^{n-1}}{c_{n-2}} \, V_t(t \xh) ,
\]
for $\xh \in S$ when $n \geq 3$, and defining
\[
U(\xh) = t \, U_t(t \xh) + \log t , \qquad V(\xh) = t \, V_t(t \xh) + \log t ,
\]
for $\xh \in S$ when $n=2$. Define a constant
\[
E=
\begin{cases}
(t^{n-1}/c_{n-2}) E_K(t) , & n \geq 3 , \\
t E_K(t) + \log t , & n = 2 .
\end{cases}
\]
One could equally well use $E_B(t)$ on the right side of this definition, since we are assuming in this part of the argument that $E_K(t)=E_B(t)$. Note that $E>0$ when $n \geq 3$.

We claim $U$ and $V$ satisfy the hypotheses of \autoref{conj:neumann}, for the sets $K/t$ and $B/t$. Indeed, by \autoref{th:upperu} we know $U$ is superharmonic on $S$ and harmonic on $S \setminus (K/t)$ with $U=E$ on $K/t$ except perhaps on a subset of inner $(n-1)$-capacity zero, and $\partial U/\partial \nu = 0$ on $\partial S$ due to the Neumann condition \eqref{eq:kernelneumann} satisfied by the kernel. As $|x|=r \to \infty$, the decay estimates on $U_t$ in \autoref{pr:potentialdecay} imply that
\[
U(x,z) =
\begin{cases}
\dfrac{1}{r^{n-2}} + O \! \left( \dfrac{1}{r^{n-1}} \right) & \text{if $n \geq 3$,} \\[1em]
  \log \dfrac{1}{r} + O \! \left( \dfrac{1}{r} \right) & \text{if $n=2$,}
\end{cases}
\]
and that the gradient satisfies
\[
\nabla U(x,z) = - \frac{b_n}{r^{n-1}} \, (x/r,0) + O\!\left( \frac{1}{r^{n}}\right) ,
\]
where $b_2=1$ and $b_n=n-2$ for $n \geq 3$. The function $V$ satisfies the same estimates as $r \to \infty$, and is superharmonic on $S$ and harmonic on $S \setminus (B/t)$ with $V=E$ on $B/t$ and $\partial V/\partial \nu = 0$ on $\partial S$.

Since \autoref{conj:neumann} is assumed to hold true, we obtain from it that $\int_S (\partial_z U)^2\,dxdz > \int_S (\partial_z V)^2 \, dxdz$, where the inequality is strict because $K$ does not equal the union of a ball in $\Rn$ with a set of inner $(n-1)$-capacity zero. The inequality rescales back to the original strip in terms of the original potentials to give that
\[
\int_{S(t)} (\partial_z U_t)^2\,dxdz > \int_{S(t)} (\partial_z V_t)^2 \, dxdz .
\]
Hence the energy derivative formula in \autoref{EKderivDirichlet} says that
\[
\big( t E_K(t) \big)^\prime > \big( t E_B(t) \big)^\prime
\]
at this particular $t$-value. Since $E_K(t)=E_B(t)$, it follows that $E_K^{{\,}\prime}(t) > E_B^{{\,}\prime}(t)$, providing the desired contradiction and hence completing the proof of \autoref{th:main}. \qed

The proof above assumes \autoref{conj:neumann}. While that conjecture remains open, we prove some closely related properties in the next section.

\section{\texorpdfstring{\bf Extremality of the ball for $L^p$-norms in the strip}{Extremality of the ball for Lp-norms in the strip}}
\label{sec:starfninequalities}
The comparison results in this section provide a certain moral support to \autoref{conj:neumann}, where one wants to compare the $L^2$ norms of the $z$-derivatives of the potentials.
We show that if $K$ and $B$ have the same energy in a strip $S(t)$, then the potential associated with $B$ is larger, in a suitably integrated sense, than the potential associated with $K$. The precise statements appear below in \autoref{th:Jvustar} and \autoref{co:integralmeans}. Also, \autoref{co:vsymmdecr} deduces the intuitively natural property that the potential for $B$ is symmetric decreasing on each horizontal slice.

We start by defining the symmetrization of sets and functions. Given a measurable set $M$ in $\Rn$, write $M^\#$ for the open ball centered at the origin having the same measure as $M$. In particular, if $M$ has measure zero then $M^\#$ is the empty set, and if $M$ has infinite measure then $M^\#=\Rn$.

Let $f$ be a Lebesgue measurable, extended-real valued function on $\Rnp$ whose superlevel sets have finite measure on almost every horizontal slice, meaning that for almost every $z \in \R$ one has
\begin{equation} \label{eq:superlevel}
|\{ f(\cdot,z)>\lambda \}| < \infty \qquad \text{for all $\lambda > \essinf f$.}
\end{equation}
The symmetric decreasing rearrangement of $f$ with respect to the horizontal variables is then defined by
\[
f^\#(x,z) = \sup \left\{ \lambda : x \in \{ f(\cdot,z) > \lambda \}^\# \right\} , \qquad x \in \Rn ,
\]
where this definition makes sense for almost every $z \in \R$. Following the terminology of Baernstein \cite[Chapter 6]{B19}, we call $f^\#$ the $(n,n+1)$-Steiner symmetrization of $f$.

The key facts are that, for almost every $z$, the rearrangement $x \mapsto f^\#(x,z)$ is equidistributed with $x \mapsto f(x,z)$, is radially symmetric (constant on each sphere $|x| = r$), and is decreasing and right continuous on rays in $\Rn$ that emanate from the origin; see \cite[Proposition 1.30]{B19}. Also, $f^\#$ is measurable on $\Rnp$ by \cite[pp.{\,}183-184]{B19}; that material assumes $f \geq 0$, but the nonnegativity assumption is not used in proving measurability.

The next lemma verifies conditions on strip potentials that will allow us to apply symmetrization.
\begin{lemma}[Superlevel set integrability for strip potentials] \label{le:superlevel}
Suppose $K$ is a compact nonempty subset of the strip $S(t)\subset \Rnp$, where $n\geq 2$ and $t>0$ are fixed. Let $\mu$ be a probability measure on $K$ and write $U$ for the potential of $\mu$ with respect to the kernel $G_t$ on the strip. Then $U$ satisfies the superlevel set condition \eqref{eq:superlevel} for every $z \in [-t,t]$ and furthermore the following integrals of $U$ over superlevel sets in $\Rnp$ are finite:
\begin{equation} \label{eq:BS} 
\int_{S(t)} (U(\xh)-\lambda)_+ \, d\xh < \infty
\end{equation}
for all $\lambda>0$ (when $n \geq 3$) or all $\lambda>-\infty$ (when $n=2$).
\end{lemma}
\begin{proof}
First, $U$ is lower semicontinuous and hence is measurable on the horizontal slice at height $z$. When $n \geq 3$ the potential is positive and $U(x,z)$ decays to zero as $|x| \to \infty$ according to \autoref{pr:potentialdecay}, uniformly with respect to $z$. Hence \eqref{eq:superlevel} holds with $\essinf U = 0$. When $n=2$, the potential tends to $-\infty$ as $|x| \to \infty$ according again to \autoref{pr:potentialdecay}, and so \eqref{eq:superlevel} holds with $\essinf U = -\infty$. In each case, the behavior of $U$ at infinity justifies inequality \eqref{eq:BS}.
\end{proof}
Our comparison theorem will utilize the operator
\[
(Jf)(r,z) = \int_{\B^n(r)} f(x,z) \, dx , \qquad r \geq 0, \quad z \in \R ,
\]
which integrates over an $n$-dimensional ball centered in the slice at height $z$. This $Jf$ is a function on $\R_+ \times \R$, where $\R_+ = (0,\infty)$. The $J$ operator applied to $f^\#$ gives
\[
J(f^\#)(r,z) = \int_{\B^n(r)} f^\#(x,z) \, dx .
\]
For our potentials, a useful equivalent characterization is that
\begin{equation} \label{eq:starcharacterization}
J(U^\#)(r,z) = \sup \left\{ \int_A U(x,z) \, dx : A \subset \Rn \text{ is bounded and } |A|=|\B^n(r)| \right\} ,
\end{equation}
or in other words, that $J(U^\#)=U^\s$ in the star-function notation of Baernstein \cite[Chapter 9]{B19}. For this characterization, see \cite[(9.85)]{B19}, which ultimately relies on \cite[Proposition 9.2(b)]{B19}.

One cannot generally restrict attention to bounded sets $A$ in characterizing $J(f^\#)$, but we may do so in \eqref{eq:starcharacterization} because $U(x,z)$ is greater than $\essinf U$ and tends to that value as $|x| \to \infty$, so that points near infinity cannot contribute to sets that approach the supremum in \eqref{eq:starcharacterization}.

We arrive next at the main result of the section. For notational simplicity, we break with the convention used in the rest of the paper and employ lowercase letters for the equilibrium potentials, writing
\[
u = U_t \qquad \text{and} \qquad v=V_t
\]
for the equilibrium potentials of the sets $K$ and $B$, respectively, on the strip $S(t)$. The set $K$ is assumed to lie in the hyperplane $\Rn$.
\begin{theorem}[Comparison of $Jv$ and $J(u^\#)$] \label{th:Jvustar}
Let $K$ be a compact subset of $\Rn, n \geq 2$, that has positive $(n-1)$-capacity. Fix $t>0$, and suppose $B \subset \Rn$ is a closed ball centered at the origin having the same energy as $K$ in the strip $S(t) \subset \Rnp$, meaning $E_K(t)=E_B(t) < \infty$. If $u$ and $v$ are the equilibrium potentials associated with these energies of $K$ and $B$, respectively, then
\[
J(u^\#) \leq Jv
 \]
 for all $r \geq 0, z \in [-t,t]$.
\end{theorem}
Before proceeding to prove the theorem, we pause to develop two corollaries that are not needed elsewhere and yet help shed light on the properties of $v$.
\begin{corollary}[$v$ is symmetric decreasing on each slice] \label{co:vsymmdecr}
Fix $t>0$, let $B \subset \Rn$ be a closed ball centered at the origin and write $v$ for the equilibrium potential associated with the energy $E_B(t)$ on the strip $S(t) \subset \Rnp$. Then $v=v^\#$, so that $v$ is symmetric decreasing on each horizontal slice.
\end{corollary}
\begin{proof}[Proof of \autoref{co:vsymmdecr}]
Taking $K=B$ in \autoref{th:Jvustar} yields that $J(v^\#) \leq Jv$. The reverse inequality $J(v^\#) \geq Jv$ always holds, by choosing $A=\B^n(r)$ in the characterization \eqref{eq:starcharacterization}. Hence $Jv = J(v^\#)$. Since $v$ is radially symmetric on each horizontal slice (by rotational symmetry of $B$ and uniqueness of the equilibrium measure), and $v$ is continuous when $z \neq 0$, it follows by differentiating $Jv$ and $J(v^\#)$ with respect to $r$ that $v=v^\#$ when $z \neq 0$. Thus $v(x,z)$ is symmetric decreasing as a function of $x$ on each slice with $z \neq 0$.

Now consider $z=0$. We know $v$ is continuous outside $B$ but we do not know a priori about continuity on $B$. Observe that for each $x \in \Rn$ and $(y,0) \in B$, the Newtonian kernel $|(x,z)-(y,0)|^{-(n-1)}$ increases monotonically as $z \to 0$ from either above or below, and so by applying monotone convergence to the Newtonian (singular) part of the kernel $G_t$ we find that $v(x,z) \to v(x,0)$ as $z \to 0$. Since $v(x,z)$ is radially symmetric decreasing when $z \neq 0$, so must be $v(x,0)$. Any decreasing function of the radial variable that is lower semicontinuous is necessarily right-continuous. Hence $v=v^\#$.

Incidentally, $v(x,0)=E_B(t)$ for all $(x,0) \in B \setminus Z_t$, by \autoref{th:upperu} applied with $K=B$. That exceptional set $Z_t$ has inner capacity zero and hence $n$-dimensional Lebesgue measure zero (\autoref{le:Zmzero}), and so the fact that $v(x,0)$ is symmetric decreasing forces $v$ to equal the constant $E_B(t)$ everywhere on $B$, except perhaps at the outer radius of the ball.
\end{proof}
\begin{corollary}[Integral means comparison] \label{co:integralmeans}
Under the assumptions of \autoref{th:Jvustar}, if $\Phi : \R \to [0,\infty)$ is convex and increasing then
\begin{equation} \label{eq:Phiintegrals}
\int_\Rn \Phi(u(x,z)) \, dx \leq \int_\Rn \Phi(v(x,z)) \, dx
\end{equation}
for each $z \in [-t,t]$. In particular, when $n \geq 3$ we have
\begin{equation} \label{eq:Lp}
\int_\Rn u(x,z)^p \, dx \leq \int_\Rn v(x,z)^p \, dx , \qquad p \in \left( \frac{n}{n-2},\infty \right) ,
\end{equation}
for each $z \in [-t,t]$.
\end{corollary}
The $\Phi$ integrals in the corollary might both equal $+\infty$, in which case the inequality between them would still be true, although uninformative.
\begin{proof}[Proof of \autoref{co:integralmeans}]
By equimeasurability of $u$ and $u^\#$ on each slice, the goal \eqref{eq:Phiintegrals} is equivalent to having
\begin{equation} \label{eq:Phiintegralssharp}
\int_\Rn \Phi(u^\#(x,z)) \, dx \leq \int_\Rn \Phi(v(x,z)) \, dx
\end{equation}
for each $z \in [-t,t]$. This inequality follows from $J(u^\#) \leq Jv$ by a standard majorization technique, but we do not know a reference that quite fits the current situation, in which $x \mapsto u(x,z)$ is nonintegrable on $\Rn$ (decaying only like $|x|^{-(n-2)}$ at infinity, when $n \geq 3$). For that reason, we now present the majorization argument.

First, using monotone convergence it suffices to prove \eqref{eq:Phiintegralssharp} under the additional assumption that $\Phi(\lambda)=0$ for all $\lambda$ sufficiently large negative. For such a $\Phi$ one has the Lebesgue--Stieltjes representation $\Phi(\lambda) = \int_\R (\lambda-\kappa)_+ \, d(\Phi^\prime(\kappa-))$, noting that the function $\Phi^\prime(\kappa-)$ is increasing by convexity of $\Phi$.
Thus we need only prove \eqref{eq:Phiintegralssharp} for functions of the form $\Phi(\lambda)=(\lambda-\kappa)_+$, with $\kappa \in \R$ fixed. When $n \geq 3$, we may further suppose $\kappa > 0$ since both sides of inequality \eqref{eq:Phiintegralssharp} equal $+\infty$ when $\kappa \leq 0$, by nonintegrability of $x \mapsto u(x,z)$ on $\Rn$. Now, given $z \in [-t,t]$, we choose $r \geq 0$ to be the smallest value such that $u^\#(x,z) \leq \kappa$ when $|x|=r$. Then
\begin{align*}
\int_\Rn (u^\#(x,z)-\kappa)_+ \, dx
& = \int_{\B^n(r)} (u^\#(x,z)-\kappa) \, dx \\
& \leq \int_{\B^n(r)} (v(x,z)-\kappa) \, dx && \text{since $J(u^\#) \leq Jv$ by \autoref{th:Jvustar}} \\
& \leq \int_\Rn (v(x,z)-\kappa)_+ \, dx ,
\end{align*}
which is inequality \eqref{eq:Phiintegralssharp} for $\Phi(\lambda)=(\lambda-\kappa)_+$.

Note that the assumption $p>n/(n-2)$ and the decay rate for the potential in \autoref{pr:potentialdecay} ensure that the $L^p$ norms in \eqref{eq:Lp} are in fact finite when $n \geq 3$.
\end{proof}

In the remainder of this section we develop tools to help prove \autoref{th:Jvustar}. The proof is then given in the next section.

\subsection*{Commutation relation}
The adjoint of $J$ is defined by
\begin{equation} \label{eq:JTis}
(J^T \psi)(x,z) = \int_{|x|}^\infty \psi(r,z)\,dr , \qquad x \in \Rn , \quad z \in \R ,
\end{equation}
which indeed operates as an $L^2$-adjoint since
\begin{equation} \label{Baernstein9.86}
\int_\R \int_{\R_+} Ju(r,z) \, \psi(r,z) \,dr dz = \int_\R \int_\Rn u(x,z) \, J^T \psi(x,z)\, dxdz
\end{equation}
by Fubini, whenever $u \in L^1_{loc}(\Rnp)$ and $\psi \in C_c(\R_+ \times \R)$.

We borrow this $J$ and $J^T$ notation from Baernstein \cite[Chapter 10]{B19}, whose approach to comparison theorems for partial differential equations emphasizes commutation relations such as in the following lemma.

Define a differential operator $\deltastar$ acting on functions $\phi(r,z)$ by
\begin{equation} \label{eq:Jmoment2}
\begin{split}
\deltastar \phi
& = r^{n-1} \partial_r \left( r^{1-n} \partial_r \phi \right) + \partial_{zz} \phi \\
& = \partial_{rr} \phi - \frac{n-1}{r} \partial_r \phi + \partial_{zz} \phi ,
\end{split}
\end{equation}
which has formal adjoint
\[
\deltastart \psi = \partial_{rr} \psi + \partial_r \! \left( \frac{n-1}{r} \psi \right) + \partial_{zz} \psi .
\]

\begin{lemma}[Commutation relation; Baernstein \protect{\cite[formula (9.90)]{B19}}] \label{le:commute}
The $J$ operator ``commutes'' with the Laplacian $\Delta$ on $\Rnp$, meaning that
\[
\deltastar J f = J \Delta f
\]
whenever $f \in C^2(\Rnp)$. More generally, if $f \in L^1_{loc}(\Rnp)$ then the commutation relation holds in the sense of distributions:
\begin{equation} \label{eq:distributioncommutation}
\int_\R \int_{\R_+} (\deltastart \psi) Jf \, dr dz = \int_\R \int_\Rn (\Delta J^T \psi) f \, dxdz
\end{equation}
for all test functions $\psi \in C_c^2(\R_+ \times \R)$.
\end{lemma}
\begin{proof}
For the commutation relation $\deltastar J f = J \Delta f$ when $f$ is $C^2$-smooth on $\Rnp$, see Baernstein \cite[p.{\,}352]{B19}. Or else for a more detailed proof, consult \cite[Lemma 5.1]{CL24}.

The proof below for the non-smooth, distributional version of the formula follows the argument for \cite[(9.90)]{B19}, but spares the reader from having to trace back various formulas through that chapter.

Note that an adjoint commutation relation
\begin{equation} \label{eq:adjointcommutation}
J^T(\deltastart \psi) = \Delta(J^T \psi)
\end{equation}
holds for $\psi \in C^2_c(\R_+ \times \R)$; to prove this formula, multiply the smooth commutation relation by $\psi$, integrate over $\R_+ \times \R$, and move each operator from $f$ onto $\psi$ by means of the adjoint operators, and finally use the arbitrariness of $f \in C^2(\Rnp)$ to obtain \eqref{eq:adjointcommutation} pointwise. In this argument, $J^T \psi(x,z)$ is obviously $C^2$-smooth away from the $z$-axis (that is, for $x \neq 0$), and is also $C^2$-smooth on a neighborhood of of the $z$-axis, because $\psi(r,z)$ vanishes for all sufficiently small $r$ and so $J^T \psi(x,z)$ is independent of $x$ for all sufficiently small $|x|$. Thus it is permissible to take the Laplacian of $J^T \psi$.

Applying the $J$-adjoint relation to the left side of the distributional commutation relation \eqref{eq:distributioncommutation} and then using \eqref{eq:adjointcommutation} yields the right side of \eqref{eq:distributioncommutation}, thus finishing the proof of the lemma.
\end{proof}

\subsection*{\texorpdfstring{Behavior of $u^\#$ near infinity}{Behavior of the rearranged potential near infinity}}
In the build-up to proving \autoref{th:Jvustar}, next we show that the symmetric decreasing rearrangement of the potential on each horizontal plane satisfies the same decay estimates as the potential itself.

\begin{proposition}[Rearranged potential decay at infinity] \label{pr:rearrangedecay}
Suppose $K$ is a nonempty compact subset of the strip $S(t)\subset \Rnp$, where $n\geq 2$ and $t>0$ are fixed. Let $\mu$ be a probability measure on $K$ and write $U$ for the potential of $\mu$. Then as $r=|x| \to \infty$, the symmetric decreasing rearrangement of the potential on horizontal slices decays according to
\[
U^\#(x,z) =
\begin{cases}
\dfrac{c_{n-2}}{t} \dfrac{1}{r^{n-2}} + O \! \left( \dfrac{1}{r^{n-1}} \right) & \text{when $n \geq 3$,} \\[1em]
 \dfrac{1}{t} \log \dfrac{1}{r} + O \! \left( \dfrac{1}{r} \right) & \text{when $n=2$,}
\end{cases}
\]
where the remainder estimates are uniform with respect to $z \in [-t,t]$.
\end{proposition}
\begin{proof}
First, consider the case $n \geq 3$. From \autoref{pr:potentialdecay}, there are $d,R_1>0$ such that
\[
\dfrac{c_{n-2}}{t} \dfrac{1}{r^{n-2}} - \dfrac{d}{r^{n-1}}
\leq U(x,z) \leq \dfrac{c_{n-2}}{t} \dfrac{1}{r^{n-2}} + \dfrac{d}{r^{n-1}}
\]
for all $r=|x|>R_1$ and $z\in[-t,t]$. We will prove that the same bounds hold also for $U^\#$, for sufficiently large $r$. The strategy is to construct lower and upper bounding functions whose rearrangements have suitable decay. More precisely, we will construct functions $\ell$ and $h$ such that
\[
0 < \ell(x,z) \leq U(x,z) \leq h(x,z)
\]
for all $x,z$, and
\[
\ell^\#(x,z)=\dfrac{c_{n-2}}{t} \dfrac{1}{r^{n-2}} - \dfrac{d}{r^{n-1}}, \qquad h^\#(x,z) = \dfrac{c_{n-2}}{t} \dfrac{1}{r^{n-2}} + \dfrac{d}{r^{n-1}}
\]
for sufficiently large $r=|x|$. Rearrangement preserves order, and so
\[
\ell^\#(x,z) \leq U^\#(x,z) \leq h^\#(x,z),
\]
which gives the desired estimates for $U^\#$. See \autoref{fig:rearranged-decay} for an illustration of the bounding functions $\ell$ and $h$ and their rearrangements.

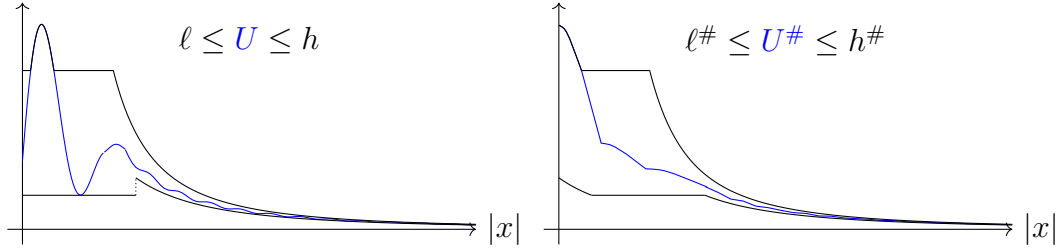
\begin{figure}
\begin{tikzpicture}[scale=1,declare function = {
f(\t)=2.3/(\t)^2-0.5/(\t)^3+(0.09/(\t^1.7))*sin(900*\t); 
g(\t)=0.9+2.5*(1-\t)*sin(300*\t); 
h(\t)= 1.07+(0.07/(\t))*sin(900*\t-700);
}]
\draw [->] (-0.2,0) -- (6,0) node[right] {$|x|$};
\draw [->] (0,-0.2) -- (0,3);
\draw[domain=0:1.07,variable=\t,samples=100, blue]
  plot (\t,{g(\t)});
\draw[domain=1.12:1.39, variable=\t,blue]
  plot ({(1.02*\t-0.06)},{h(\t)});
\draw[domain=1.37:6,variable=\t,samples=100,blue]
  plot ({(\t-0.01)},{f(\t)});
\draw[domain=0:1.2,variable=\t,samples=150]
  plot (\t,{max(2.1,g(\t))});
\draw[domain=1.2:6,variable=\t,samples=100]
  plot (\t,{2.2/(\t)^2+1/(\t)^3});
\draw [-] (0,0.45) -- (1.5,0.45);
\draw[domain=1.5:6,variable=\t,samples=90]
  plot (\t,{2.2/(\t)^2-1/(\t)^3});
\draw[densely dotted] (1.5,0.45) -- (1.5,0.68);

\draw (3,2.5) node {$\ell \leq {\color{blue}U}\leq h$};
\end{tikzpicture}
\begin{tikzpicture}[scale=1,declare function = {
f(\t)=2.3/(\t)^2-0.5/(\t)^3+(0.09/(\t^1.7))*sin(900*\t); 
g(\t)=0.9+2.5*(1-\t)*sin(300*\t); 
h(\t)= 1.07+(0.07/(\t))*sin(900*\t-700);
k(\t)= ((3/(\t)^2+(0.09/(\t^1.8))*sin(900*\t)))*0.75;
}]
\draw [->] (-0.2,0) -- (6,0) node[right] {$|x|$};
\draw [->] (0,-0.2) -- (0,3);
\draw[blue,smooth] (0,2.7) .. controls (0.1,2.65) .. (0.3,2.1) -- (0.56,1.14) .. controls (0.7,1.12).. (0.9, 0.99) -- (1.15, 0.8) .. controls (1.4,0.79) .. (1.96,0.56);
\draw[domain=2:6, variable=\t, blue]
  plot ({(\t-0.06)},{k(\t)});
\draw (0,2.7) .. controls (0.1,2.65) .. (0.3,2.1);
\draw (0.3, 2.1) -- (1.2,2.1);
\draw[domain=1.2:6,variable=\t,samples=100]
  plot (\t,{2.2/(\t)^2+1/(\t)^3});
\draw[domain=1.5:1.93,variable=\t,samples=90]
  plot ({(\t-1.5)},{2.2/(\t)^2-1/(\t)^3});
\draw [-] (0.43,0.45) -- (1.93,0.45);
\draw[domain=1.93:6,variable=\t,samples=90]
  plot (\t,{2.2/(\t)^2-1/(\t)^3});
  
\draw (3,2.5) node {$\ell^\#\leq {\color{blue}U^\# }\leq h^\#$};
\end{tikzpicture}

\caption{Qualitative illustration of the graphs of the lower and higher functions $\ell$ and $h$ and their rearrangements, showing cross-sections at some fixed $z$-value.} \label{fig:rearranged-decay}
\end{figure}

 For the upper bound define
\[
h(x,z)=
\begin{cases}
\max\left(U(x,z), \dfrac{c_{n-2}}{t} \dfrac{1}{R_1^{n-2}} + \dfrac{d}{R_1^{n-1}} \right), & |x| \leq R_1 \\
 \dfrac{c_{n-2}}{t} \dfrac{1}{|x|^{n-2}} + \dfrac{d}{|x|^{n-1}}, & |x| > R_1.
\end{cases}
\]
This function satisfies $U(x,z)\leq h(x,z)$ for all $x,z$ and so $U^\#(x,z)\leq h^\#(x,z)$. The rearrangement $h^\#(x,z)$ is well defined because $h$ is positive and vanishes as $|x|\to\infty$. Moreover, for $|x| > R_1$ we have $h^\#(x,z)= \frac{c_{n-2}}{t} \frac{1}{r^{n-2}} + \frac{d}{r^{n-1}}$ because that expression is decreasing.

Next we construct the lower bound $\ell$. Note that
\[
\dfrac{c_{n-2}}{t} \dfrac{1}{r^{n-2}} - \dfrac{d}{r^{n-1}}
\]
is positive decreasing for $r>\frac{d (n-1) t}{c_{n-2}(n-2)}$, and approaches $0$ as $r\to \infty$. Let $R_2=\max(R_1,\frac{d (n-1) t}{c_{n-2}(n-2)})$. The potential $U$ is positive and lower semicontinuous on the compact set $\{(x,z) \, : \, |x|\leq R_2, |z|\leq t\}$, and so there is a positive constant $\alpha$ such that $U(x,z) \geq \alpha$ on that set.
The function
\[
\ell(x,z)=
\begin{cases}
\alpha, & |x| \leq R_2, \\
\dfrac{c_{n-2}}{t} \dfrac{1}{|x|^{n-2}} - \dfrac{d}{|x|^{n-1}} , & |x| >R_2,
\end{cases}
\]
satisfies $0<\ell(x,z) \leq U(x,z)$ for all $(x,z)$, and so also $\ell^\# (x,z) \leq U^\#(x,z)$. Moreover, $\ell^\#(x,z)= \frac{c_{n-2}}{t} \frac{1}{r^{n-2}} - \frac{d}{r^{n-1}}$ for all $r$ large enough that the expression is less than $\alpha$. That finishes the proof when $n>2$.

The case $n=2$ follows the same line of argument, with only minor adjustments. The main difference is that the potentials and bounds approach $-\infty$ instead of $0$ as $|x|\to \infty$. There are $d,R_1>0$ such that
\[
\dfrac{1}{t}\log \dfrac{1}{|x|} - \dfrac{d}{|x|}
\leq U(x,z) \leq \dfrac{1}{t}\log \dfrac{1}{|x|} + \dfrac{d}{|x|}
\]
for all $r=|x|>R_1$ and $z\in[-t,t]$. The upper estimate $U^\# \leq h$ follows as above after making the analogous choice for $h$:
\[
h(x,z)=
\begin{cases}
\max\left(U(x,z), \dfrac{1}{t}\log \dfrac{1}{R_1} + \dfrac{d}{R_1} \right), & |x| \leq R_1 \\
 \dfrac{1}{t}\log \dfrac{1}{|x|} + \dfrac{d}{|x|}, & |x| > R_1.
\end{cases}
\]
For the lower bound, one notes that there is a radius $R_2\geq R_1$ such that $\dfrac{1}{t}\log \dfrac{1}{r} - \dfrac{d}{r}$ is decreasing for $r\geq R_2$ and approaches $-\infty$ as $r\to \infty$. The potential $U$ is bounded below by some $\alpha$ on the compact set $\{(x,z) \, : \, |x|\leq R_2, |z|\leq t\}$. The lower estimate $\ell \leq U^\#$ then follows as above with the following choice for $\ell$:
\[
\ell(x,z)=
\begin{cases}
\alpha, & |x| \leq R_2, \\
\dfrac{1}{t}\log \dfrac{1}{|x|} - \dfrac{d}{|x|}, & |x| >R_2.
\end{cases}
\]
\end{proof}

\section{\bf Proof of \autoref{th:Jvustar}}
\label{sec:proofofJvustar}

\subsection*{Overview and symmetrization literature}
The $\star$-function symmetrization technique of Baernstein \cite[Chapter 10]{B19} will be adapted to the current setting. His technique permits the right side of Poisson's equation $- a_n \Delta u = \mu$ to be a measure, which is crucial for our application because the equilibrium measure $\mu$ is supported in $K \subset \Rn$ by hypothesis and so is singular with respect to Lebesgue measure in $\Rnp$.

The symmetrization theorems of Talenti \cite{T76,T16}, Alvino, Lions and Trombetti \cite{ALT91} and other authors are based on level set decomposition methods, and as a result do not apply to singular measures.

In the literature, the closest result to what we need is perhaps \cite[Theorem 10.15]{B19}. It does not seem possible to apply that theorem directly, because when $n=2$ our $u$ and $v$ diverge logarithmically at infinity and so fail the nonnegativity hypothesis in the theorem. More seriously, the hypothesis $\mu^\star \leq J\nu$ of that theorem fails since our $\mu^\#$ (defined in Step 2 below) concentrates at radius $0$ while $\nu$ does not.

We surmount the latter obstacle in Step 2 by proving the subsolution property of the $\star$-function not on all of space but rather on a smaller domain $\mathcal{D}$ obtained by removing a periodic array of intervals that correspond to the original ball $B$. The ``missing'' intervals are then handled in Step 5 by using that $u \leq E_K(t) = E_B(t) = v$ on the ball $B$.

Symmetrization theorems generally require Dirichlet boundary conditions. The Neumann condition on the strip boundary is eliminated in the following proof by periodizing to obtain a problem on all of space. Neumann conditions have been handled in certain other situations in the literature. For example, on annular shells a symmetrization theorem for Poisson's equation is given in \cite[Theorem 10.20]{B19} by building on work of Langford \cite{L15a,L15b}. Robin boundary conditions have been investigated by Alvino, Nitsch and Trombetti \cite{ANT23} and Langford \cite{L23}.

\subsection*{Goal of the proof}
To prove \autoref{th:Jvustar} we must show $\phi \geq 0$, where
\[
\phi=Jv-J(u^\#) .
\]

\subsection*{\texorpdfstring{Step 1. $\phi$ is continuous}{Step 1. The comparison function is continuous}}
For continuity, we want to show
\[
\lim_{(r^\prime,z^\prime) \to (r,z)} \phi(r^\prime,z^\prime) = \phi(r,z)
\]
whenever $r \geq 0, z \in [-t,t]$. In view of the local Lipschitz continuity with respect to $r$ in \autoref{le:lipschitz} below, we may fix $r^\prime=r$ and simply show continuity in the $z$-variable. Thus it will suffice to prove
\begin{equation} \label{eq:Jusharptcont}
\lim_{z^\prime \to z} J(u^\#)(r,z^\prime) = J(u^\#)(r,z)
\end{equation}
and
\begin{equation} \label{eq:Jtvcont}
\lim_{z^\prime \to z} (Jv)(r,z^\prime) = (Jv)(r,z) .
\end{equation}
These limits are established in Steps 1a, 1b, 1c below.
\begin{lemma}[Lipschitz continuity of $Ju$ and $J(u^\#)$] \label{le:lipschitz} If $R>0$ then $r \mapsto (Ju)(r,z)$ and $r \mapsto J(u^\#)(r,z)$ are Lipschitz continuous for $r \in [0,R]$, with Lipschitz constant that may depend on $R$ but is independent of $z \in [-t,t]$.
\end{lemma}
\begin{proof}
When $n \geq 3$, the potential $u$ is bounded by \autoref{th:upperu}. Hence $u^\#$ is bounded too, from which Lipschitz continuity of $r \mapsto (Ju)(r,z)$ and $r \mapsto J(u^\#)(r,z)$ follows easily for $r$ in the interval $[0,R]$.

Now suppose $n=2$. The upper boundedness and local lower boundedness of $u$ in \autoref{th:upperu} implies that $u^\#$ is bounded above and is locally bounded below on $\overline{S(t)}$. Hence $r \mapsto (Ju)(r,z)$ and $r \mapsto J(u^\#)(r,z)$ are locally Lipschitz continuous, with constant that is independent of $z$.
\end{proof}

\subsection*{\texorpdfstring{Step 1a: $J(u^\#)$ is continuous when $z \neq 0$}{Step 1a: Continuity away from the central slice}}
To prove \eqref{eq:Jusharptcont} when $z$ is nonzero, observe that if $z,z^\prime \in (0,t]$ then for all bounded sets $A \subset \Rn$ with $|A|=|\B^n(r)|$ one has
\begin{align*}
\left| \int_A u(x,z) \, dx - \int_A u(x,z^\prime) \, dx \right|
& \leq |A| \lVert \partial_z u \rVert_{L^\infty(\Rn \times [z,z^\prime])} |z-z^\prime| .
\end{align*}
Here the $L^\infty$ norm is finite by smoothness of the potential away from $K$ and the gradient decay estimate in \autoref{pr:potentialdecay}. This inequality holds for all $A$, and so the characterization \eqref{eq:starcharacterization} implies that
\[
\left| J(u^\#)(r,z) - J(u^\#)(r,z^\prime) \right| \leq |\B^n(r)| \lVert \partial_z u \rVert_{L^\infty([z,z^\prime] \times \Rn)} |z-z^\prime| ,
\]
which tends to zero as $z^\prime \to z$. The same argument works when $z, z^\prime \in [-t,0)$. Hence \eqref{eq:Jusharptcont} holds as desired, when $z \neq 0$.

\subsection*{\texorpdfstring{Step 1b: $J(u^\#)$ is continuous at $z = 0$}{Step 1b: Continuity at the central slice}} In this case we want to show that
\begin{equation} \label{eq:Jusharptcont0}
\lim_{z \to 0} J(u^\#)(r,z) = J(u^\#)(r,0)
\end{equation}
for each $r>0$. (When $r=0$, both sides equal $0$.)

To start with, the pointwise convergence $u(x,0) = \lim_{z \to 0} u(x,z)$ holds for almost every $x \in \Rn$ since $u$ is continuous on $\overline{S(t)}$ except on a set $Z_t \subset K \subset \Rn$ that has $n$-dimensional Lebesgue measure zero, by \autoref{th:upperu} and \autoref{le:Zmzero}. Further, $u$ is locally bounded by \autoref{th:upperu}. Thus if $A \subset \Rn$ is bounded with $|A|=|\B^n(r)|$ then dominated convergence ensures that
\begin{align}
\int_A u(x,0) \, dx
& = \int_A \lim_{z \to 0} u(x,z) \, dx \notag \\
& = \lim_{z \to 0} \int_A u(x,z) \, dx \label{eq:continuityatzero}
\end{align}
and so
\[
\int_A u(x,0) \, dx \leq \liminf_{z \to 0} J(u^\#)(r,z)
\]
by the characterization \eqref{eq:starcharacterization}. Taking the supremum over all sets $A$, we deduce that
\[
J(u^\#)(r,0) \leq \liminf_{z \to 0} J(u^\#)(r,z) .
\]

Thus to establish \eqref{eq:Jusharptcont0}, we want an inequality in the reverse direction, namely
\begin{equation} \label{eq:Jusharptcont00}
\limsup_{z \to 0} J(u^\#)(r,z) \leq J(u^\#)(r,0) .
\end{equation}
Consider $z \in [-t,t]$. For $\e>0$, denote the $\e$-neighborhood of $K$ in $\Rn$ by
\[
K(\e) = \{ x \in \Rn : \dist(x,K) < \e \} ,
\]
and let ${\mathcal C}(\e) = K(\e) \setminus K$ be the $\e$-collar. Since $u \leq E_K(t)$ everywhere by \autoref{th:upperu}, with equality almost everywhere on $K$ with respect to $n$-dimensional Lebesgue measure, we see that
\begin{align*}
\int_A u(x,z) \, dx
& \leq \int_{A \setminus K(\e)} u(x,z) \, dx + \int_{A \cap K} u(x,0) \, dx + |A \cap {\mathcal C}(\e)| \, E_K(t) \\
& \leq \int_{A \setminus K(\e)} u(x,0) \, dx + \int_{A \cap K} u(x,0) \, dx + |{\mathcal C}(\e)| \, E_K(t) + |A| O_\e(|z|)
\end{align*}
where to estimate the first term we used that $u(x,z)=u(x,0)+O_\e(|z|)$ because the gradient $\nabla u$ on the closed set $(\Rn \setminus K(\e)) \times [-t,t]$ is smooth and bounded (the boundedness for large $x$ holding by \autoref{pr:potentialdecay}). Combining the two integrals on the right side therefore yields that
\begin{equation} \label{eq:alongtheway}
\int_A u(x,z) \, dx
\leq \int_{A \setminus {\mathcal C}(\e)} u(x,0) \, dx + |{\mathcal C}(\e)| \, E_K(t) + |A| O_\e(|z|) .
\end{equation}

When $n \geq 3$, the potential $u$ is nonnegative and so we may increase the integral on the right of \eqref{eq:alongtheway} to obtain
\[
\int_A u(x,z) \, dx
\leq \int_A u(x,0) \, dx + |{\mathcal C}(\e)| \, E_K(t) + |A| O_\e(|z|) .
\]
By taking the supremum with respect to the set $A$ in the characterization \eqref{eq:starcharacterization}, we deduce that
\[
J(u^\#)(r,z)
\leq J(u^\#)(r,0) + |{\mathcal C}(\e)| \, E_K(t) + |\B^n(r)| O_\e(|z|) .
\]
Letting $z \to 0$ and then $\e \to 0$ establishes \eqref{eq:Jusharptcont00}, since the collar ${\mathcal C}(\e)$ shrinks to the empty set.

When $n=2$ we must argue more carefully, since $u$ can take negative values. Require $\e$ to be small enough that $|{\mathcal C}(\e)| < |\B^n(r)|$. Write $r(A,\e)$ and $r(\e)$ for the radii such that
\begin{align*}
|\B^n(r(A,\e))| & = |A \setminus {\mathcal C}(\e)| , \\
|\B^n(r(\e))| & = |A| - |{\mathcal C}(\e)| = |\B^n(r)| - |{\mathcal C}(\e)| ,
\end{align*}
so that
\[
0 < r(\e) \leq r(A,\e) \leq r .
\]
Clearly $r(\e) \to r$ as $\e \to 0$, since the collar vanishes in the limit. Hence from \eqref{eq:alongtheway} and the characterization \eqref{eq:starcharacterization} we find
\begin{align*}
\int_A u(x,z) \, dx
& \leq J(u^\#)(r(A,\e),0) + |{\mathcal C}(\e)| \, E_K(t) + |\B^n(r)| O_\e(|z|) \\
& \leq J(u^\#)(r,0) + O(r-r(\e)) + |{\mathcal C}(\e)| \, E_K(t) + |\B^n(r)| O_\e(|z|)
\end{align*}
by the local Lipschitz continuity in \autoref{le:lipschitz}. The right side is now independent of $A$. On the left, taking the supremum with respect to $A$ and then letting $z \to 0$ gives that
\[
\limsup_{z \to 0} J(u^\#)(r,z) \leq J(u^\#)(r,0) + O(r-r(\e)) + |{\mathcal C}(\e)| \, E_K(t) .
\]
Letting $\e \to 0$ yields \eqref{eq:Jusharptcont00}, finishing the proof for Step 1b that $J(u^\#)$ is continuous at $z = 0$.

\subsection*{\texorpdfstring{Step 1c: $Jv$ is continuous}{Step 1c: Continuity of Jv}} $Jv(r,z)$ is continuous when $z \neq 0$, because $v(x,z)$ is continuous off the set $B$. For continuity of $Jv$ at $z=0$ we need only prove \eqref{eq:Jtvcont} with $z=0$, which follows by applying \eqref{eq:continuityatzero} with $u=v$ and $A=\B^n(r)$.

Steps 1a, 1b and 1c together establish the continuity of $\phi$ for Step 1.

\subsection*{\texorpdfstring{Step 2. $\phi$ is a $\deltastar$-supersolution}{Step 2. The comparison function is a supersolution}}
Let us extend the equilibrium potential $u$ and associated equilibrium measure $\mu$ to $\Rnp$ as in the proof of \autoref{th:upperu}. Recall $K \subset \Rn$ by hypothesis in \autoref{th:Jvustar} so that points in $K$ have the special form $\yh=(y,0)$. Thus on $K$, the transformations $\rho_j$ are simply the vertical translations $\rho_j(y,0)=(y,2tj)$. Hence the set $\widetilde{K}$ in the proof of \autoref{th:upperu} is obtained simply by periodizing $K$ in the $z$-direction:
\[
\widetilde{K} = \cup_{j \in \Z} \left( K + (0,2tj) \right) .
\]
Define the measure $\widetilde{\mu}$ to be the periodic extension of $\mu$ from $K$ to $\widetilde{K}$.

When $n \geq 3$, the extension of $u$ is
\begin{equation} \label{eq:fund}
u(\xh) = \int_{\widetilde{K}} \frac{1}{|\xh-\yh|^{n-1}} \, d\widetilde{\mu}(\yh) , \qquad \xh = (x,z) \in \Rnp ,
\end{equation}
by applying \eqref{eq:extendedpotential} to the current situation. Thus the extension of $u$ is the Newtonian potential of the periodically extended measure $\widetilde{\mu}$. Clearly $u$ is $2t$-periodic with respect to $z$.

When $n=2$, the extension formula for the potential is
\begin{equation}
u(\xh)
= \int_{\widetilde{K}} \left( \frac{1}{|\xh-\yh|} - \frac{1_{ \{ w \neq 0 \} }}{|w|} \right) d\widetilde{\mu}(\yh) + \frac{1}{t} (\gamma - \log(4t)) , \qquad \xh = (x,z) \in \R^3 , \label{eq:fund2}
\end{equation}
where as usual $\yh = (y,w)$. Note $w=2tj$ when $\yh$ lies in the $j$-th component of $\widetilde{K}$. One verifies straightforwardly that $u$ is $2t$-periodic with respect to $z$.

One similarly extends $v$ and its associated measure $\nu$ periodically in the $z$-direction, for each $n \geq 2$. The distributional Laplacian of a potential gives its associated measure and so \eqref{eq:fund} and \eqref{eq:fund2} imply that in all dimensions
\[
-a_n \Delta u = \widetilde{\mu} , \qquad -a_n \Delta v = \widetilde{\nu} .
\]

Next we show $\phi$ is a $\deltastar$-supersolution, for the operator $\deltastar$ that was defined in \eqref{eq:Jmoment2}. Note that $\phi=Jv-J(u^\#)$ is continuous on the right halfplane $\R_+ \times \R$, by Step 1 above.

For $v$, \autoref{le:commute} says that $\deltastar J v = J \Delta v$ in the sense of distributions. Thus formula \eqref{eq:distributioncommutation} yields that
\begin{equation} \label{eq:super1}
- a_n \int_\R \int_{\R_+} (\deltastart \psi) Jv \, dr dz = \int_\Rnp J^T \psi \, d\widetilde{\nu}
\end{equation}
for all test functions $\psi \in C_c^2(\R_+ \times \R)$, where on the right side of \eqref{eq:distributioncommutation} we used that $-a_n \Delta v = \widetilde{\nu}$ and that $J^T \psi$ is a valid test function, being $C^2$-smooth and compactly supported on $\Rnp$ (as follows from the definition \eqref{eq:JTis} of $J^T$).

To develop an analogous inequality for $u$, first we define a rearranged measure $\widetilde{\mu}^\#$ on $\Rnp$. Recall that $\widetilde{\mu}$ is supported in the copies of $K \subset \Rn$ translated vertically by integer multiples of $2t$. In particular, $\widetilde{\mu}$ is singular with respect to Lebesgue measure on $\Rnp$. Hence in accordance with Baernstein \cite[Definition 9.21]{B19}, we define $\widetilde{\mu}^\#$ by sweeping the $\widetilde{\mu}$-measure to the $z$-axis, resulting in a sum of delta measures at heights $2tj$ in $\Rnp$:
\[
\widetilde{\mu}^\# = \sum_{j \in \Z} \delta_{(0,2tj)} .
\]
The claimed analogue of \eqref{eq:super1} says that $-a_n\deltastar J(u^\#) \leq J(\widetilde{\mu}^\#)$ in the distributional sense, meaning that
\begin{equation} \label{eq:super2}
- a_n \int_\R \int_{\R_+} (\deltastart \psi) J(u^\#) \, dr dz \leq \int_\Rnp J^T \psi \, d\widetilde{\mu}^\#
\end{equation}
provided the test function $\psi \in C_c^2(\R_+ \times \R)$ is nonnegative. This claim is a special case of Baernstein's ``subharmonicity'' theorem \cite[Theorem 9.23]{B19}. To reduce the technical burden of reading that result, we prove \eqref{eq:super2} more directly as follows:
\begin{align}
- a_n \int_\R \int_{\R_+} (\deltastart \psi) J(u^\#) \, dr dz
& = - a_n \int_\Rnp (J^T \deltastart \psi) u^\# \, dx dz && \text{by \eqref{Baernstein9.86}} \notag \\
& = - a_n \int_\Rnp (\Delta J^T \psi) u^\# \, dx dz && \text{by \eqref{eq:adjointcommutation}} \notag \\
& \leq \int_\Rnp J^T \psi \, d\widetilde{\mu}^\# \label{eq:presubharmonicity}
\end{align}
by Baernstein's ``pre-subharmonicity'' theorem for $(n,n+1)$-Steiner symmetrization \cite[Theorem 9.22]{B19}, whose hypotheses we proceed to verify. When $n \geq 3$ the theorem applies as follows:
\begin{itemize}
\item take the domain to be $\Omega = \Rn \times (-s,s)$ with $s$ chosen large enough that the function $J^T \psi$ is supported in $\Omega$, and note that $\Omega^\# = \Omega$,
\item $u>0$ is lower semicontinuous on $\Omega$,
\item $u(x,z) \to 0$ whenever $z \to z_0 \in (-s,s)$ and $|x| \to \infty$,
\item $-a_n \Delta u = \widetilde{\mu}$ in $\Omega$ in the sense of distributions,
\item $\widetilde{\mu}$ is positive and singular with respect to Lebesgue measure in $\Omega$,
\item $\widetilde{\mu}(\Omega) < \infty$ since $s<\infty$.
\end{itemize}
When $n=2$ the pre-subharmonicity theorem can be applied after adding a suitable constant to $u$, as follows:
\begin{itemize}
\item choose $s$ large enough that the cylinder $\B^n(s) \times (-s,s)$ contains $K$ and the support of $J^T \psi$,
\item take the domain to be $\Omega = \{ (x,z) \in \Rn \times (-s,s) : u(x,z) > -c \}$,
\item choose $c$ large enough that $u>-c$ on the closure of the cylinder (remembering $u > -\infty$ is lower semicontinuous and so is bounded below on compact sets); hence $\Omega$ and $\Omega^\#$ contain the cylinder and thus contain $K$ and the support of $J^T \psi$,
\item $u+c>0$ is lower semicontinuous on $\Omega$,
\item $\Omega$ is bounded since $u(x,z) \to -\infty$ as $|x| \to \infty$ by \autoref{pr:potentialdecay} with $n=2$,
\item $u(x,z)+c = 0$ whenever $z \in (-s,s)$ and $(x,z) \in \partial \Omega$,
\item $-a_n \Delta (u+c) = \widetilde{\mu}$ in $\Omega$ in the sense of distributions,
\item $\widetilde{\mu}$ is positive and singular with respect to Lebesgue measure in $\Omega$,
\item $\widetilde{\mu}(\Omega) < \infty$ since $s<\infty$.
\end{itemize}
Using this set-up, one may readily verify the hypotheses of the pre-subharmonicity theorem \cite[Theorem 9.22]{B19} for the function $g=J^T \psi$, which is nonnegative and Steiner symmetric (radially symmetric decreasing on each slice) as required, because $\psi$ is assumed nonnegative. The conclusion of the pre-subharmonicity theorem then yields inequality \eqref{eq:presubharmonicity}.

Subtracting formulas \eqref{eq:super1} and \eqref{eq:super2} gives that
\begin{equation} \label{eq:phisupersolution}
a_n \int_\R \int_{\R_+} (\deltastart \psi) \phi \, dr dz \leq \int_\Rnp J^T \psi \, d(\widetilde{\mu}^\#-\widetilde{\nu})
\end{equation}
provided $\psi \geq 0$.

Now write $\beta$ for the radius of the ball $B$ in $\Rn$ and consider the domain
\[
{\mathcal D} = \big( \R_+ \times \R \big) \setminus \big( \cup_{j \in \Z} \, (0,\beta] \times \{ 2tj \} \big)
\]
that is obtained from the right halfplane by removing intervals of length $\beta$ at each height that is a multiple of $2t$. The domain ${\mathcal D}$ is illustrated in \autoref{fig:domainD}.

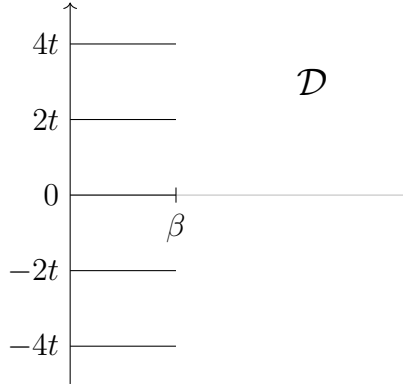
\begin{figure}
\begin{tikzpicture}
\def\radius{1.4}

\draw[->,gray!50] (0,0) -- (4.5,0);
\draw[->] (0,-2.5) -- (0,2.55);

\draw[-] (0,2) node[left] {$4t$}-- (\radius,2) ;
\draw[-] (0,1) node[left] {$2t$}-- (\radius,1) ;
\draw[-] (0,0) node[left] {$0$} -- (\radius,0);
\draw[-] (0,-1) node[left] {$-2t$}-- (\radius,-1);
\draw[-] (0,-2) node[left] {$-4t$}-- (\radius,-2);

\draw[-] (\radius,-0.1) node[below] {$\beta$} -- (\radius,0.1);

\draw (3.2,1.5) node {\large$\mathcal{D}$};
\end{tikzpicture}
\caption{The domain $\mathcal{D}$ obtained from the right halfplane by removing intervals of length $\beta$ at heights that are multiples of $2t$. \label{fig:domainD}}
\end{figure}

We claim $\phi$ is a $\deltastar$-supersolution on ${\mathcal D}$, meaning that
\[
\deltastar \phi \leq 0 \quad \text{on ${\mathcal D}$}
\]
in the distributional sense with respect to nonnegative test functions $\psi \in C^2_c({\mathcal D})$. In view of inequality \eqref{eq:phisupersolution} it suffices to show that
\[
\int_\Rnp J^T \psi \, d(\widetilde{\mu}^\#-\widetilde{\nu}) = 0
\]
for all $\psi \in C^2_c({\mathcal D})$. By periodicity and a partition of unity in the $z$-direction, we may suppose without loss of generality that the compactly supported function $\psi$ is supported in the subregion of ${\mathcal D}$ where $|z| < 2t$. In this subregion, the measures $\widetilde{\mu}^\#$ and $\widetilde{\nu}$ are supported at height $z=0$ in the ball $B$, with $\widetilde{\mu}$ there being just a delta measure $\delta_{(0,0)}$ at the origin and $\widetilde{\nu}$ being the equilibrium measure $\nu$ on the ball $B$. When $|x| \leq \beta$, the integrand $J^T \psi(x,0)$ is constant by definition of $J^T$ since $\psi$ equals zero on the interval $[0,\beta] \times \{ 0 \}$, which lies outside ${\mathcal D}$. Hence
\[
\int_\Rnp J^T \psi \, d(\widetilde{\mu}^\#-\widetilde{\nu}) = J^T \psi(0,0) (\delta_{(0,0)}(B)-\nu(B)) = 0
\]
as desired, because $\delta_{(0,0)}$ and $\nu$ each has total measure $1$ on $B$.

\subsection*{\texorpdfstring{Step 3. $\phi_\e$ grows at infinity and is a strict supersolution in ${\mathcal D}$}{Step 3. The perturbed comparison function is a strict supersolution}}
We perturb $\phi$ by defining
\[
\phi_\e(r,z)=\phi(r,z)+\e (r^n+r)
\]
for $\e>0$. This $\phi_\e$ grows to infinity as $r \to \infty$, as follows.

If $n \geq 3$ then $v$ and $u^\#$ decay like $O(1/r^{n-2})$, by \autoref{pr:potentialdecay} for $v$ and \autoref{pr:rearrangedecay} for $u^\#$. After integrating in spherical coordinates on the $n$-dimensional slice at height $z$, one sees that $Jv$ and $J(u^\#)$ are bounded by $O(r^2)$. That bound is dominated by $\e r^n$ since $n \geq 3$, implying that $\phi_\e(r,z) \to \infty$ as $r \to \infty$.

If $n=2$ then the difference $v-u^\#$ decays like $O(1/r)$, again by \autoref{pr:potentialdecay} for $v$ and \autoref{pr:rearrangedecay} for $u^\#$, and so $J(v-u^\#)$ is bounded by $O(r)$, which is dominated by $\e r^n$. Thus once again $\phi_\e(r,z) \to \infty$ as $r \to \infty$.

Notice that $\phi_\e$ is a strict distributional supersolution, with
\[
\deltastar \phi_\e \leq - \e (n-1) r^{-1} < 0
\]
in ${\mathcal D}$ by Step 2 since direct computation shows that $\deltastar (r^n+r) = -(n-1)/r$.

\smallskip \noindent
\emph{Comment.} The two perturbing terms $r^n$ and $r$ fulfill distinct roles in Step 3. Namely, $r^n$ ensures that $\phi_\e$ grows at infinity, while the $r$ term forces $\deltastar \phi_\e$ to be negative.

\subsection*{\texorpdfstring{Step 4. $\phi_\e$ attains its minimum on $\partial {\mathcal D}$}{Step 4. The minimum occurs on the boundary}}

From the continuity in Step 1 and the growth at infinity in Step 3, it follows that the continuous function $\phi_\e$ achieves its minimum in the closed right halfplane at some finite point $(r_\e,z_\e)$ with $0 \leq r_\e < \infty$ and $z_\e \in \R$. Since $\phi_\e$ is a strict distributional supersolution, its minimum cannot occur in ${\mathcal D}$, by the strong minimum principle; see for example \cite[Proposition 10.7]{B19}, noting that as required there the operator $\deltastar$ has no zero-th order terms. Thus the minimum point for $\phi_\e$ occurs on $\partial {\mathcal D}$, meaning that either $r_\e \in [0,\beta]$ and $z_\e$ is an integer multiple of $2t$, or else $r_\e=0$.

\subsection*{\texorpdfstring{Step 5. $\phi$ is nonnegative}{Step 5. The comparison function is nonnegative}}
If $r_\e=0$ in Step 4, then trivially $\phi_\e(0,z_\e)=0$.

If instead $r_\e \in [0,\beta]$ and $z_\e$ is an integer multiple of $2t$, then we may suppose $z_\e=0$ by periodicity. By \autoref{th:upperu} and \autoref{le:Zmzero}, the equilibrium potential for the ball satisfies $v(x,0) = E_B(t)$ for almost every $x$ with $|x| \leq \beta$, while \autoref{th:upperu} says that the potential for $K$ satisfies $u(x,0) \leq E_K(t)$ for all $x$. Since $E_B(t)=E_K(t)$ by hypothesis in \autoref{th:Jvustar}, it follows that $v(x,0) \geq u^\#(x,0)$ for almost every $x$ with $|x| \leq \beta$. Hence at the minimum point we find
\[
\phi_\e(r_\e,0)=\int_{\B^n(r_\e)} (v(x,0)-u^\#(x,0)) \, dx + \e (r_\e^n+r_\e) \geq 0 .
\]

Thus in each case the minimum value is nonnegative, so that $\phi_\e(r,z) \geq 0$ for all $r,z$. Letting $\e \to 0$ yields that $\phi(r,z) \geq 0$, proving \autoref{th:Jvustar}.

\section*{\bf Acknowledgments}
Richard Laugesen's research was supported by grants from the Simons Foundation (\#964018) and the National Science Foundation ({\#}2246537). We are grateful to Thomas Ransford for pointing us to the relevant portion of Fuglede's paper.

\section*{\bf AI usage statement}
The authors developed the manuscript without use of artificial intelligence. During final revisions,
OpenAI's ChatGPT/Codex with the GPT-5.6 Sol model assisted with a detailed
critical review in which it checked
calculations, logical dependencies, assumptions, notation, exposition and
citations. The review suggested some corrections and clarifications, in response to which the authors added a few paragraphs.

\appendix

\section{\bf Potential theory in the strip}
\label{sec:background}

Fix $n\geq 2$ throughout this appendix.

\subsection*{Inner capacity zero} A set $Z \subset \Rnp$ is said to have inner $(n-1)$-capacity zero if $\capnone(Z^\prime)=0$ for every compact $Z^\prime \subset Z$.
\begin{lemma}[Inner capacity zero implies measure zero] \label{le:Zmzero}
If $Z \subset \Rnp$ is measurable and has inner $(n-1)$-capacity zero then $Z$ has $(n+1)$-dimensional Lebesgue measure zero. If $Z \subset \Rn$ is measurable with inner $(n-1)$-capacity zero then $Z$ has $n$-dimensional Lebesgue measure zero.
\end{lemma}
The lemma is well known.
\begin{proof}
We start by proving the contrapositive of the first claim. Suppose $Z$ has positive $(n+1)$-dimensional Lebesgue measure, so that by inner regularity some compact subset $Z^\prime \subset Z$ also has positive measure. By restricting Lebesgue measure to $Z^\prime$ and normalizing, one obtains a probability measure that is easily shown to have finite $(n-1)$-energy. Hence $\capnone(Z^\prime)>0$, and so $Z$ does not have inner $(n-1)$-capacity zero.

The second claim, for $Z \subset \Rn$, is proved similarly.
\end{proof}
\begin{lemma}[Inner capacity zero does not change the energy] \label{le:Z}
Suppose $K,Z \subset \Rnp$ and that $K$ and $K \cup Z$ are compact. If $Z$ has inner $(n-1)$-capacity zero then $E_K(t)=E_{K \cup Z}(t)$ for all $t \in (0,\infty]$ large enough that the strip $S(t)$ contains $K \cup Z$.
\end{lemma}
When $t=\infty$, recall from \autoref{sec:stripenergy} that $G_\infty$ is the Newtonian kernel, $E_K(\infty)$ is the Newtonian energy $V_{n-1}(K)$ in $\Rnp$, and $\mu_\infty$ is the associated Newtonian equilibrium measure on $K$.
\begin{proof}
In the whole space case ($t=\infty$), the lemma asserts that $V_{n-1}(K) = V_{n-1}(K \cup Z)$, which is well known; see for example \cite[Lemma 22]{CL25b}. The proof of that case adapts easily to strip energies with $t<\infty$, as follows.

We may suppose $K$ and $Z$ are disjoint, by replacing $Z$ with the smaller set $Z \setminus K$. It suffices to show $E_{K \cup Z}(t) \geq E_K(t)$, since the other direction is immediate from set-monotonicity of the energy. We may suppose $E_{K \cup Z}(t)$ is finite.

Let $\mu$ be a probability measure on $K \cup Z$ having finite energy with respect to the kernel $G_t$. We claim $\mu(Z)=0$. Suppose $\mu(Z)>0$, so that $\mu(\Rn \setminus K)>0$ and hence the open set $\Rn \setminus K$ contains some compact subset $\widetilde{Z}$ with positive $\mu$-measure. Then the compact set
\[
Z^\prime=\widetilde{Z} \cap (K \cup Z) = \widetilde{Z} \cap Z \subset Z
\]
has $\mu(Z^\prime)=\mu(\widetilde{Z})>0$. Thus the restriction of $\mu$ to $Z^\prime$ has finite $G_t$-energy and positive $\mu$-measure. Hence $E_{Z^\prime}(t)$ is finite, and so $E_{Z^\prime}(\infty)=V_{n-1}(Z^\prime)$ is also finite (see \cite[Lemma 2.2]{CL26}), so that $Z^\prime$ has positive $(n-1)$-capacity, contradicting that $Z$ has inner capacity zero. Therefore $\mu(Z)=0$ as claimed, and so $\mu$ is supported on $K$, which implies that $E_{K \cup Z}(t) \geq E_K(t)$ as desired.
\end{proof}

\subsection*{Properties of equilibrium potentials} Suppose $K \subset \Rnp$ is compact and has positive $(n-1)$-capacity, so that $V_{n-1}(K)$ is finite. Let $t \in (0,\infty]$ and assume $t$ is large enough that $K \subset S(t)$. Write $\mu_t$ for an equilibrium measure associated to the energy $E_K(t)$, as in \autoref{sec:stripenergy}.

The equilibrium potential is
\[
U_t(\xh) = \int_K G_t(\xh,\yh) \, d\mu_t(\yh) , \qquad \xh \in \overline{S(t)} .
\]
As is true of all potentials with kernel $G_t$, this equilibrium potential is superharmonic, and is smooth and harmonic away from the support of $\mu_t$. According to the next theorem, the equilibrium potential has the special property that it is bounded above by the energy value.
\begin{theorem}[Fundamental theorem of equilibrium potentials] \label{th:upperu}
Under the assumptions above we have
\begin{equation} \label{eq:potentialinequality}
U_t(\xh) \leq E_K(t) , \qquad \xh \in \overline{S(t)} .
\end{equation}
Equality holds at all interior points of $K$ and also at all boundary points except perhaps on an exceptional subset $Z_t \subset \partial K$ of inner $(n-1)$-capacity zero. Inequality \eqref{eq:potentialinequality} is strict on the unbounded component of the complement $\overline{S(t)} \setminus K$. Furthermore:
\begin{itemize}
\item $U_t$ is continuous at each point in $\overline{S(t)} \setminus Z_t$,
\item $U_t$ is bounded when $n \geq 3$, and when $n=2$ it is bounded above and locally bounded below,
\item $Z_t$ is a countable union of compact sets each having $(n-1)$-capacity zero.
\end{itemize}
\end{theorem}
\begin{proof}
Consider first $n \geq 3$.

\subsubsection*{Step 1} We will show the kernel $G_t$ satisfies Frostman's maximum principle, meaning that for each probability measure $\mu$ on the arbitrary compact set $K \subset S(t)$ we have
\begin{equation} \label{frostman}
\sup_{\overline{S(t)}} U \leq \sup_K U ,
\end{equation}
where $U(\xh) = \int_K G_t(\xh,\yh) \, d\mu(\yh)$ is the potential generated by $\mu$. One could prove this principle by mimicking the Euclidean proof \cite[pp.{\,}69--72]{L72} and using the known behavior of $G_t$ at horizontal infinity. But we prefer in what follows to reduce directly to the Euclidean situation by employing the explicit formula for $G_t$.

The definition of the kernel in \autoref{sec:stripenergy} gives that
\begin{align}
U(\xh)
& = \sum_{j \in \Z} \int_K \frac{1}{|\xh-\rho_j(\yh)|^{n-1}} \, d\mu(\yh) \notag \\
& = \int_{\widetilde{K}} \frac{1}{|\xh-\yh|^{n-1}} \, d\widetilde{\mu}(\yh) \label{eq:extendedpotential}
\end{align}
where the pushforward measure
\[
\widetilde{\mu} = \sum_{j \in \Z} \mu \circ \rho_j^{-1}
\]
is supported on a set
\[
\widetilde{K} = \cup_{j \in \Z} \, \rho_j(K) \subset \Rnp
\]
that is obtained by applying the family $\{ \rho_j \}$ of translations and reflections to $K$. Notice that \eqref{eq:extendedpotential} extends $U$ to a function on all of $\Rnp$ and shows that it equals the Newtonian potential of $\widetilde{\mu}$. Frostman's maximum principle for Newtonian potentials of arbitrary measures (see Landkof \cite[Theorem 1.10]{L72} with $\alpha=2$) now implies that on $\Rnp$ we have $U \leq M := \sup_{\widetilde{K}} U$.

To prove inequality \eqref{frostman}, it remains only to show that $\sup_{\widetilde{K}} U = \sup_K U$. When taking the supremum over $\widetilde{K}$ we need only consider its subsets $K$ and $\rho_1(K)$, since \eqref{eq:rhodefn} implies that $\rho_{j+2}(\yh)=\rho_j(\yh)+(0,4t)$ and so the measure $\widetilde{\mu}$ and potential $U$ are $4t$-periodic in the vertical direction. Further, the supremum of $U$ over $\rho_1(K)$ equals the supremum over $K$ since $U \circ \rho_1 = U$ (noting that $|\rho_1(\xh)-\rho_j(\yh)|=|\xh-\rho_{1-j}(\yh)|$). Thus $\sup_{\widetilde{K}} U = \sup_K U$, and so we have proved inequality \eqref{frostman} when $n \geq 3$.

\subsubsection*{Step 2}
Step 1 applies with $K$ replaced by the compact support of any probability measure on $K$, and thus establishes Frostman's maximum principle in Fuglede's support-based sense \cite[\S2.1]{F60}. For $n\geq3$, $G_t$ is a positive, symmetric, lower semicontinuous kernel on $X=\overline{S(t)}$. Thus Fuglede \cite[Theorem 2.4 and Remark 1]{F60} applies directly for $n\geq3$, showing that the equilibrium potential is at most $E_K(t)$ everywhere and equals $E_K(t)$ on $K$ outside an exceptional set $Z_t$ of inner $G_t$-capacity zero. The case $n=2$ is justified separately in Step 3 below. (Fuglede \cite[Lemma 2.3.1(v)]{F60} shows that his ``nearly everywhere'' condition is equivalent to the exceptional set having inner capacity zero with respect to $k=G_t$, which is equivalent by \cite[Lemma 2.2]{CL26} to inner capacity zero with respect to $G_\infty$, which is $(n-1)$-capacity zero.) The set $Z_t$ of inner capacity zero necessarily has Lebesgue measure zero in $\Rnp$ by \autoref{le:Zmzero}, and so the supermeanvalue property of the potential implies that equality holds in \eqref{eq:potentialinequality} at all interior points of $K$. Therefore $Z_t \subset \partial K$.

Lastly, on the unbounded component of $\Rnp \setminus K$, the strong maximum principle applied to $U$ yields strict inequality in \eqref{eq:potentialinequality}.


\subsubsection*{Step 3} Now consider the case $n=2$. We will again establish Frostman's maximum principle \eqref{frostman}. Then one may apply Fuglede's result on $K$, after shifting the restricted kernel by a constant, as explained at the end of this step.

The challenge in proving Frostman's principle when $n=2$ is that we no longer know the kernel is positive, due to the subtraction of the terms in the definition of $G_t$ that ensure convergence of the series. Those terms must be handled carefully in what follows.

Let us consider at first the partial sum of the measure $\widetilde{\mu}$, denoted
\[
\mu_J = \sum_{|j| \leq J} \mu \circ \rho_j^{-1} ,
\]
which is supported on the set $K_J = \cup_{|j| \leq J} \, \rho_j(K)$ that is obtained from $K$ by applying the translations and reflections $\rho_j$ for $-J \leq j \leq J$. Applying Frostman's maximum principle (Landkof \cite[Theorem 1.10]{L72} with $\alpha=2, p=3$) to the Newtonian potential $U_J(\xh) = \int_{K_J} |\xh-\yh|^{-1} \, d\mu_J(\yh)$ on $\R^3$, we get in particular that
\[
U_J(\xih) \leq \sup_{K_J} U_J , \qquad \xih \in \overline{S(t)} .
\]
This inequality can be rewritten as
\[
\int_K \sum_{|j| \leq J} |\xih-\rho_j(\yh)|^{-1} \, d\mu(\yh) \leq \sup_{\xh \in K_J} \int_K \sum_{|j| \leq J} |\xh-\rho_j(\yh)|^{-1} \, d\mu(\yh) .
\]

Choose $\ell = \ell(J)$ with $|\ell| \leq J$ such that the supremum on the right side is approached in the set $\rho_\ell(K)$. Then the preceding inequality implies for all $\xih \in S(t)$ that
\begin{align}
\int_K \sum_{|j| \leq J} |\xih-\rho_j(\yh)|^{-1} \, d\mu(\yh)
& \leq \sup_{\xh \in K} \int_K \sum_{|j| \leq J} |\rho_\ell(\xh)-\rho_j(\yh)|^{-1} \, d\mu(\yh) \label{eq:leftside} \\
& = \sup_{\xh \in K} \int_K \sum_{j \in L} |\xh-\rho_j(\yh)|^{-1} \, d\mu(\yh) \label{eq:rightside}
\end{align}
by using the definitions of $\rho_\ell$ and $\rho_j$ to check that
\[
|\rho_\ell(\xh)-\rho_j(\yh)|
=
\begin{cases}
|\xh - \rho_{j-\ell}(\yh)| & \text{if $\ell$ is even,} \\
|\xh - \rho_{\ell-j}(\yh)| & \text{if $\ell$ is odd,} \\
\end{cases}
\]
and then translating the indices in \eqref{eq:leftside} to arrive at the new index set, which is
\[
L =
\begin{cases}
\{ -J-\ell,\ldots,J-\ell \} & \text{if $\ell$ is even,} \\
\{ -J+\ell,\ldots,J+\ell \} & \text{if $\ell$ is odd.} \\
\end{cases}
\]
Note that $L$ consists of $2J+1$ consecutive integers, and $0 \in L$.

Write $\delta = \diam(K)$. On the left side of \eqref{eq:leftside} we subtract
\[
\sum_{1 \leq |j| \leq J} \frac{1}{2t(|j|+1)+\delta}
\]
and on the right side of \eqref{eq:rightside} we subtract the smaller quantity
\[
\sum_{j \in L \setminus \{ 0 \}} \frac{1}{2t(|j|+1)+\delta} ,
\]
so that
\begin{equation} \label{eq:renorm1}
\begin{split}
& \int_K \left[ \frac{1}{|\xih-\yh|} + \sum_{1 \leq |j| \leq J} \! \left( \frac{1}{|\xih-\rho_j(\yh)|} - \frac{1}{2t(|j|+1)+\delta} \right) \right] \! d\mu(\yh) \\
& \leq \sup_{\xh \in K} \int_K \left[ \frac{1}{|\xh-\yh|} + \sum_{j \in L \setminus \{ 0 \}} \! \left( \frac{1}{|\xh-\rho_j(\yh)|} - \frac{1}{2t(|j|+1)+\delta} \right) \right] \! d\mu(\yh) .
\end{split}
\end{equation}
If $\xh,\yh \in K$ and $j \in \Z$ then
\[
|\xh-\rho_j(\yh)| \leq |\xh-\yh| + |\yh-\rho_j(\yh)| \leq \delta + 2t(|j|+1) .
\]
Thus each summand on the right side of \eqref{eq:renorm1} is nonnegative, and so the right side increases if we expand the range of summation to include all $j \neq 0$. Hence
\begin{equation} \label{eq:renorm2}
\begin{split}
& \int_K \left[ \frac{1}{|\xih-\yh|} + \sum_{1 \leq |j| \leq J} \! \left( \frac{1}{|\xih-\rho_j(\yh)|} - \frac{1}{2t(|j|+1)+\delta} \right) \right] \! d\mu(\yh) \\
& \leq \sup_{\xh \in K} \int_K \left[ \frac{1}{|\xh-\yh|} + \sum_{j \neq 0} \! \left( \frac{1}{|\xh-\rho_j(\yh)|} - \frac{1}{2t(|j|+1)+\delta} \right) \right] \! d\mu(\yh) .
\end{split}
\end{equation}
The right side of \eqref{eq:renorm2} is independent of $J,L,\ell$. We let $J \to \infty$ on the left side and then on both sides we add the finite constant
\[
\sum_{j \neq 0} \! \left( \frac{1}{2t(|j|+1)+\delta} - \frac{1}{|2tj|} \right) ,
\]
thereby obtaining that
\[
\int_K G_t(\xih,\yh) \, d\mu(\yh) \leq \sup_{\xh \in K} \int_K G_t(\xh,\yh) \, d\mu(\yh)
\]
for all $\xih \in S(t)$. This last inequality proves Frostman's maximum principle \eqref{frostman} for $n=2$. To complete the proof of Step 3, we apply Fuglede's result as in Step 2, with one additional modification as follows.

To apply Fuglede's result despite the possible negative values of the kernel $G_t$ on the noncompact strip, restrict $G_t$ to the compact set $K\times K$. This restricted kernel is bounded below, so we may choose a constant $C$ such that $k_C=G_t+C$ is positive on $K\times K$. Fuglede's compact-space case \cite[Theorem 2.4 and Remark 1]{F60}, applied with $X=K$ and kernel $k_C$, gives $U_t+C=E_K(t)+C$ nearly everywhere on $K$ and $U_t+C\leq E_K(t)+C$ on $\operatorname{supp}\mu_t$. The support-based maximum principle just proved extends the latter bound to the whole strip. Adding $C$ changes neither the equilibrium measure nor which compact subsets have finite energy, and so Fuglede's exceptional set has inner $G_t$-capacity zero. By \cite[Lemma 2.2]{CL26}, it has inner $(n-1)$-capacity zero as asserted in Step 2.

\subsubsection*{Step 4} Continuity of $U_t$ is clear for $\xh$ in the complement of the support of $\mu_t$, and so in particular $U_t$ is continuous on $\overline{S(t)} \setminus K$. Continuity also holds at each $\xh \in K \setminus Z_t$ since
\[
E_K(t) = U_t(\xh) \leq \liminf_{\yh \to \xh} U_t(\yh) \leq \limsup_{\yh \to \xh} U_t(\yh) \leq E_K(t) ,
\]
where the first equality holds as proved above in Steps 1, 2, 3, using that $\xh \notin Z_t$; and the first inequality holds by lower semicontinuity of the superharmonic function $U_t$; the second inequality is obvious; and the third inequality follows by \eqref{eq:potentialinequality}.

\subsubsection*{Step 5} We know from \eqref{eq:potentialinequality} that $U_t$ has upper bound $E_K(t)$. When $n \geq 3$, nonnegativity of $G_t$ provides the obvious lower bound of $0$.

Suppose $n=2$. We observed just after the definition \eqref{eq:rieszkernelone} that the kernel $G_t(\xh,\yh)$ is bounded below provided $\xh$ and $\yh$ lie in some bounded set, which holds in particular if $\xh$ lies in a bounded subset of $\overline{S(t)}$ and $\yh \in K$. Thus $U_t(\xh)$ is locally bounded below.

\subsubsection*{Step 6} The exceptional set is
\[
Z_t = \{ \xh \in K : U_t(\xh) < E_K(t) \} = \cup_{k=1}^\infty Z_t^{(k)}
\]
where
\[
Z_t^{(k)} = \{ \xh \in K : U_t(\xh) \leq E_K(t)-1/k \}
\]
is a sublevel set of the potential on $K$. Each $Z_t^{(k)}$ is compact by lower semicontinuity of $U_t$, and necessarily has $(n-1)$-capacity zero because $Z_t$ has inner capacity zero. Thus $Z_t$ is a countable union of compact sets each having $(n-1)$-capacity zero.
\end{proof}

\bibliographystyle{plain}

\vspace*{-4pt}

\end{document}